\documentclass[a4paper,10pt]{amsart}

\usepackage{amsmath}
\usepackage{amsfonts}
\usepackage{amssymb}
\usepackage{amsthm}
\usepackage{graphicx}
\usepackage{subfigure}
\usepackage{geometry}
\usepackage[utf8]{inputenc}
\usepackage{float}
\usepackage{tikz}
\usetikzlibrary{matrix}     

\newtheorem{theorem}{Theorem}[section]
\newtheorem{lemma}[theorem]{Lemma}

\numberwithin{equation}{section}

\newcommand{\ds}{\displaystyle}

\def\R{\mathbb R}
\def\Z{\mathbb Z}

\def\pa{\partial}

\def\taue{\tau^{\varepsilon}}
\def\ue{u^{\varepsilon}}
\def\ve{v^{\varepsilon}}
\def\we{w^{\varepsilon}}
\def\taub{\overline{\tau}}
\def\ub{\overline{u}}
\def\wb{\overline{w}}
\def\vb{\overline{v}}
\def\rhoe{\rho^{\varepsilon}}
\def\je{j^{\varepsilon}}
\def\rhob{\overline{\rho}}
\def\jb{\overline{j}}
\def\sumi{\sum_{i\in \Z}}

\def\Dxx{{D}_{xx}}
\def\tDxx{\tilde{D}_{xx}}
\def\tDtx{\tilde{D}_{tx}}

\title[Numerical convergence rate for a diffusive limit of
  hyperbolic systems]
{Numerical convergence rate for a diffusive limit of
  hyperbolic systems: $p$-system with damping}

\author[C. Berthon]{Christophe Berthon}
\address{Universit\'e de Nantes \\ Laboratoire de Math\'ematiques Jean Leray, CNRS UMR 6629 \\
2 rue de la Houssini\`ere, BP 92208 \\ 44322 Nantes, France}
\email{christophe.berthon@univ-nantes.fr}

\author[ M. Bessemoulin-Chatard]{Marianne Bessemoulin-Chatard}
\address{Universit\'e de Nantes \\ Laboratoire de Math\'ematiques Jean Leray, CNRS UMR 6629 \\
2 rue de la Houssini\`ere, BP 92208 \\ 44322 Nantes, France}
\email{marianne.bessemoulin@univ-nantes.fr}

\author[H. Mathis]{H\'el\`ene Mathis}
\address{Universit\'e de Nantes \\ Laboratoire de Math\'ematiques Jean Leray, CNRS UMR 6629 \\
2 rue de la Houssini\`ere, BP 92208 \\ 44322 Nantes, France}
\email{helene.mathis@univ-nantes.fr}

\begin{document}







\begin{abstract}
This paper deals with diffusive limit of the $p$-system with damping 
and its approximation by an Asymptotic Preserving (AP) Finite Volume
scheme. Provided the system is endowed with an entropy-entropy flux
pair, we give the convergence rate of classical solutions of the
$p$-system with damping towards the smooth solutions of the porous
media equation using a relative entropy method. Adopting a
semi-discrete scheme, we establish that the convergence rate is
preserved by the approximated solutions. Several numerical experiments
illustrate the relevance of this result.
\end{abstract}

\maketitle


\section{Introduction}
The present work is devoted to analyze the behavior of numerical
sche\-mes, within some asymptotic regimes, when approximating the
solutions of the $p$-system with damping. The system under
consideration reads 
\begin{equation}\label{p_sys}
\left\{\begin{aligned}
& \pa_{t}\tau-\pa_{x}v=0,\\
& \pa_{t}v+\pa_{x}p(\tau)=-\sigma\,v,
\end{aligned}\right. \quad (x,t) \in \R \times \R_{+},
\end{equation}
where $\tau>0$ stands for the specific volume of gas away from zero and $v\in\R$ is
the velocity while $\sigma>0$ denotes the friction parameter. The
pressure law $p(\tau)$ fulfills the following assumptions:
\begin{equation} \label{hyp_p}
\begin{aligned}
& p \in \mathcal{C}^{2}(\mathbb{R}_{+}^{*}),
 \quad p(\tau) > 0, \quad p'(\tau)<0,\\
& \mbox{if $\tau\ge c>0$ then there exists $m$ such that 
  $p(\tau)\ge m>0$ and $p'(\tau)\le-m<0$.}
\end{aligned}
\end{equation}

The solution $w={}^t(\tau,v)$ is assumed to belong to the following
phase space
$$
\Omega=\left\{
w={}^t(\tau,v); ~ \tau>0,v\in\R \right\}.
$$

In addition, in order to rule out unphysical solutions, the system
(\ref{p_sys}) is endowed with an entropy inequality given by
\begin{equation}\label{entropy_p_sys}
\pa_{t}\eta(\tau,v)+\pa_{x}\psi(\tau,v)\leq-\sigma\,v^{2}
\leq 0,
\end{equation}
where the entropy function is given by
$$
\eta(\tau,v)=\frac{v^2}{2}-P(\tau).
$$
The quantity $-P(\tau)$ denotes an internal energy and is
defined by
\begin{equation}\label{defP}
P(\tau)=\int_{\tau_\star}^\tau p(s) \,ds,
\end{equation}
where we have set $\tau_\star>0$ an arbitrary fixed reference specific
volume. In (\ref{entropy_p_sys}), the function $\psi$ is the
entropy flux function defined as follows:
\begin{equation}\label{fluxpsys}
\psi(\tau,u)=u\,p(\tau).
\end{equation}

The study of long time asymptotic for hyperbolic systems of
conservation laws, as (\ref{p_sys}), goes back to the work of Hsiao
and Liu \cite{Hsiao1992}. They consider the isentropic Euler system
with damping which solutions tend to those of the nonlinear porous
media equation time asymptotically. Using the existence of
self-similar solutions of the limit parabolic equations proved in
\cite{vDuyn_Peletier76,vDuyn_Peletier77}, they provide convergence
rates in $\|(w-\bar w)(t)\|_{L^\infty}=O(1) t^{-1/2}$ for smooth
solutions away from zero. Here, $\wb={}^t(\taub,\vb)$ defines the
solution of the following parabolic-type system, the
so-called porous media equation:
\begin{equation}\label{p_sys0}
\left\{\begin{aligned}
&    \pa_{t}\taub+\ds{\frac{1}{\sigma}}\pa_{xx}p(\taub)=0,\\
&    \pa_{x}p(\taub)=-\sigma\,\ub,
\end{aligned}\right.
\qquad (x,t) \in \R \times \R_{+}.
\end{equation}

Some similar convergence rates have been
obtained by Nishihara \cite{Nishihara1996, Nishihara97}.  Then, under
proper assumptions on the initial data, Nishihara and co-authors
\cite{Nishihara00} improve the convergence rate as $\|(w-\bar
w)(t)\|_{L^\infty}=O(1) t^{-3/2}$, using energy estimates techniques.
For a more general overview, we refer to the review of Mei
\cite{Mei2010} where the author gives numerous references about
convergence results for the long-time asymptotic behavior of the
$p$-system with damping (\ref{p_sys}) including references concerning
non-linear damping and boundary effects. Let us emphasize that, in
\cite{BHN07}, the authors exhibit convergence rate in time for general
dissipative hyperbolic systems under the Shizuta-Kawashima condition \cite{kawa87}.

All the aforementioned results are based on energy estimates which are
difficult to transpose in the discrete framework.  To overcome these
difficulties, an other way to study the time-asymptotic behavior of
solutions of \eqref{p_sys} is to use an appropriate time-rescaling
(for instance, see \cite{Marcati1988,Mei2010}), here governed by a
small parameter $\varepsilon>0$. We also refer the
reader to \cite{Lions1997,Naldi1998,Naldi2000} devoted to related
works where the parameter $\varepsilon>0$ is directly proportional
to the Knudsen number and the Mach number of the kinetic model.

Here, we are concerned by solutions within asymptotic regimes governed
by long time and dominant friction. As a consequence, a small
parameter $\varepsilon>0$ scales the solutions
${}^t(\taue,\ve)$ under interest which now satisfy the
following PDE system:
\begin{equation}\label{p_sys_eps_bis}
\left\{\begin{aligned}
&      \varepsilon\,\pa_{t}\taue-\pa_{x}\ve=0,\\
&      \varepsilon\,\pa_{t}\ve+\pa_{x}p(\taue)=
      -\ds{\frac{\sigma}{\varepsilon}}\,\ve,
\end{aligned}\right.
\qquad (x,t) \in \R \times \R_{+}.
\end{equation}

Because of the dominant friction, we immediately note that the
velocity solution is in the form $\ve=\varepsilon\ue$. Therefore, in
this paper, we focus on the pair $\we={}^t(\taue,\ue)\in\Omega$ solution of the
system given by
\begin{equation}\label{p_sys_eps}
\left\{\begin{aligned}
&    \pa_{t}\taue-\pa_{x}\ue=0,\\
&    \varepsilon^{2}\,\pa_{t}\ue+\pa_{x}p(\taue)=-\sigma\,\ue,
\end{aligned}\right.
\qquad (x,t) \in \R \times \R_{+},
\end{equation}
supplemented by the following entropy inequality
\begin{equation}\label{entropy_p_sys_eps}
\pa_{t}\eta^\varepsilon(\taue,\ue)+\pa_{x}\psi(\taue,\ue)\leq-\sigma\,(\ue)^{2}
\leq 0,
\end{equation}
where we have set
\begin{equation}\label{entropypsys}
\eta^{\varepsilon}(\tau,u)=\varepsilon^{2}\frac{u^{2}}{2}-P(\tau).
\end{equation}

From now on, let us underline that, in the limit of $\varepsilon$ to
zero, the solutions $\we={}^t(\taue,\ue)$ of (\ref{p_sys_eps})
converge, in a sense to be prescribed, to the solutions
$\bar{w}={}^t(\taub,\ub)$ of (\ref{p_sys0}).

Considering the behavior of the solutions of (\ref{p_sys_eps}) to the
solutions of (\ref{p_sys0}), we study the convergence of the solutions
of a hyperbolic system endowed with a stiff source term to the
solution of a parabolic problem.

Next, the existence of an entropy-entropy flux pair
$(\eta^\varepsilon,\psi)$, associated with (\ref{p_sys_eps}), where
$\eta^\varepsilon\in C^2(\Omega)$ is a strictly convex function, turns out to
be an essential ingredient in the analysis of the convergence from
$\we$ to $\bar{w}$ as $\varepsilon$ goes to zero.
Indeed, we can define the relative entropy $\eta(\we|w)$ 
of the system \eqref{p_sys_eps} which corresponds to a first order
Taylor expansion of $\eta^\varepsilon$ around
a smooth solution $\bar w$ of \eqref{p_sys0}:
\begin{equation}
 \label{eq:RE}
 \eta(\we|\bar w) = \eta^\varepsilon(\we) - \eta^\varepsilon(\bar w) 
              - \nabla\eta^\varepsilon(\bar w) \cdot (\we-\bar w ),
\end{equation}
where $\we$ is a (classical) solution of \eqref{p_sys_eps}.
Thanks to the convexity of $\eta^\varepsilon$, the relative entropy 
$\eta(\we|\bar w)$ behaves like $\|\we-\bar w\|^2_{L^2(\mathbb R)}$.

The notion of relative entropy for hyperbolic systems of conservation
laws goes back to the works of DiPerna \cite{DiP79} and Dafermos
\cite{Daf79}.  It allows to prove a stability result for classical
solutions in the class of entropy weak solutions, see
\cite{dafermosBook} for a condensed proof.

In \cite{Tza05}, Tzavaras applies a similar relative entropy technique
to study the convergence of the classical solutions of hyperbolic
systems with stiff relaxation towards smooth solutions of the limit
hyperbolic systems.  Thanks to the quadratic behavior of the relative
entropy, one can control the distance between the relaxation dynamics
and the equilibrium solutions, leading to stability and convergence
results.  Based on the same ideas, Lattanzio and Tzavaras address in
\cite{Lattanzio2013} the case of diffusive relaxation. They focus on
several hyperbolic systems with diffusive relaxation of type
\eqref{p_sys_eps}.  Under some regularity assumptions on the pressure
law, they provide convergence rate in $\varepsilon^4$.
Recently in \cite{CT2016}, the authors extend the relative entropy
method to the class of hyperbolic systems which are symmetrizable,
leading to similar convergence results in the zero-viscosity limit to
smooth solutions in a $L^p$ framework.

The main objective of this work is to recover the convergence rate in
$\varepsilon^4$ when both $\we$ and $\bar{w}$ are approximated by
relevant numerical schemes. From a numerical point of view, one of
the main difficulty stays in the derivation of a suitable
discretization of (\ref{p_sys_eps}) in order to get the required
discretization of (\ref{p_sys0}) in the limit of $\varepsilon$ to
zero.

Let us set $(H_{\Delta}^{\varepsilon})$ a discretization of the
hyperbolic system \eqref{p_sys_eps}, where $\Delta$ stands for the
discretization parameter. We distinguish two types of numerical
schemes:

\begin{itemize}
\item The scheme $(H_{\Delta}^{\varepsilon})$ is said to be
  \textbf{Asymptotically Consistent} with the parabolic limit regime
  (AC) if it is consistent with the hyperbolic model \eqref{p_sys_eps}
  for all $\varepsilon>0$ and if, in the limit
  $\varepsilon \rightarrow 0$, it converges to a scheme, say
  $(P_\Delta)$, consistent with the limit parabolic model \eqref{p_sys0}.

\item The scheme $(H_{\Delta}^{\varepsilon})$ is said to be
  \textbf{Asymptotic Preserving} (AP) if it is AC and if the
  stability conditions stay admissible for all $\varepsilon> 0$.
\end{itemize}

\begin{figure}\label{diagAP}
\begin{center}
\hspace*{-0mm}
\begin{minipage}{7cm}
\begin{tikzpicture}
\matrix (m) [matrix of math nodes,row sep=4em,column sep=10em,minimum width=2em]
{(H^{\varepsilon}_{\Delta}) & (H^{\varepsilon}) \\
(P_{\Delta})  & (P)\\};
\path[-stealth]
(m-2-1) edge node [below] {$\Delta \rightarrow 0$} (m-2-2)
(m-1-1) edge node [left] {$\varepsilon \rightarrow 0$} (m-2-1)
(m-1-2) edge node [right] {$\varepsilon \rightarrow 0$} (m-2-2)
(m-1-1) edge node [above] {$\Delta \rightarrow 0$} (m-1-2);
\end{tikzpicture}
\end{minipage}
\begin{minipage}{7cm}
$(H^\varepsilon)$: Scaled hyperbolic system (\ref{p_sys_eps})\\
$(P)$: Parabolic asymptotic regime (\ref{p_sys0})\\
$(H^\varepsilon_\Delta)$: Discretization of (\ref{p_sys_eps})\\
$(P_\Delta)$: Discretization of (\ref{p_sys0})
\end{minipage}
\end{center}
\caption{Diagram of the asymptotic preserving properties}
\end{figure}
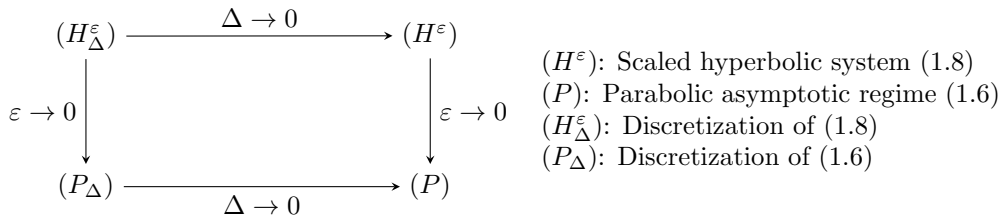

The notion of \textit{asymptotic-preserving scheme} was first
introduced by Jin \textit{et al.} in \cite{Jin1998,Jin1999} in the
context of diffusive limits for kinetic equations. Naldi and Pareschi
also proposed several numerical schemes for a two velocities kinetic
equation \cite{Naldi1998,Naldi2000}. Since these seminal articles, a
large variety of asymptotic-preserving schemes have been proposed, for
various physical models. Concerning specifically the discretization of
hyperbolic systems with source terms in the diffusive limit, Gosse and
Toscani proposed a well-balanced and asymptotic-preserving scheme for
the Goldstein-Taylor model in \cite{Gosse2002}, and then for more
general discrete kinetic models in \cite{Gosse2003}. In
\cite{Berthon2011}, Berthon and Turpault propose a modification of the
HLL scheme \cite{Harten1983} for hyperbolic systems to include source terms, and then a
correction which allows to be consistent at the diffusive limit. More recently, several works are devoted to the derivation of asymptotic-preserving schemes for 2D problems on unstructured meshes \cite{Blachere2016,Buet2012,Buet2016}.

The purpose of this article is to study the convergence rate of the
numerical scheme $(H^{\varepsilon}_{\Delta})$ towards the numerical
scheme $(P_{\Delta})$ as $\varepsilon$ tends to 0 (see Figure
\ref{diagAP}). After the work by Lattanzio and Tzavaras
\cite{Lattanzio2013}, we here adopt an error estimation given by a
relative entropy in order to exhibit the required convergence rate from
$(H^{\varepsilon}_{\Delta})$ to $(P_{\Delta})$. Indeed, in
\cite{Lattanzio2013}, the relative entropy is considered to establish
the expected convergence rate from the scaled $p$-system
(\ref{p_sys_eps}) to the porous media problem (\ref{p_sys0}). Let us
note that the relative entropies have been recently suggested in
\cite{JV06,CMS15} in order to derive suitable error estimates for
finite volume approximations of smooth solutions of nonlinear
hyperbolic systems.

The paper is organized as follows. In the next section, for the sake
of completeness, we give the main properties satisfied by the relative
entropy associated with (\ref{p_sys_eps}). More precisely, we detail
the convergence rate obtained by Lattanzio and Tzavaras
\cite{Lattanzio2013}, from the so-called $p$-system (\ref{p_sys_eps})
to the porous media equation (\ref{p_sys0}). In fact, the establishment
of this result is constructive and it will be suitably adapted to get
the expected numerical convergence rate. Section 3 concerns our main
result. By adopting a semi-discrete in space numerical scheme to
approximate the weak solutions of (\ref{p_sys_eps}), we exhibit the
convergence rate as $\varepsilon$ goes to zero, to recover a
semi-discrete approximation of the porous media equation
(\ref{p_sys0}). Moreover, the obtained convergence rate, from a
numerical point of view, exactly coincides with the one established in
\cite{Lattanzio2013} from a continuous point of view. The numerical
convergence rate is next illustrated, in the last section, performing
several numerical experiments by adopting a full discrete scheme
proposed by Jin {\it et al.} \cite{Jin1998}. The performed simulations give
an approximated convergence rate in perfect agreement with the
numerical convergence rate established in Section 3. As a consequence,
it seems that our main result is thus optimal.

\section{Convergence in the diffusive limit}

In this section, we recall the convergence result established in
\cite{Lattanzio2013} since it is useful in the forthcoming numerical development. For the sake of simplicity, the convergence statement is given by arguing smooth solutions. Such an assumption is not at all restrictive in the derivation of our main numerical result established in the next section. We refer to \cite{Lattanzio2013} to extend the following results with weak solutions.

 To exhibit the rate of convergence from
$(\taue,\ue)$, solution of (\ref{p_sys_eps}), to $(\bar{\tau},\bar{u})$,
solution of (\ref{p_sys0}), in the limit of $\varepsilon$ to zero,
Lattanzio and Tzavaras \cite{Lattanzio2013} adopt the well-known
relative entropy to define an error estimate. Considering the
$p$-system (\ref{p_sys_eps}), the relative entropy is defined by
\begin{align}
\eta^{\varepsilon}(\tau,u|\taub,\ub)&=\eta^{\varepsilon}(\tau,u)-\eta^{\varepsilon}(\taub,\ub)
-\nabla\eta^{\varepsilon}(\taub,\ub)\cdot
\left(\begin{array}{l}\tau-\taub\\u-\ub \end{array}\right),\nonumber\\
&=\frac{\varepsilon^{2}}{2}(u-\ub)^{2}-P(\tau|\taub),\label{entropyrelpsys}
\end{align}
with 
\begin{equation}\label{def-P-Tau-Taub}
P(\tau|\taub)=P(\tau)-P(\taub)-p(\taub)(\tau-\taub).
\end{equation}

This relative entropy satisfies an evolution law given in the
following statement.
\begin{lemma}\label{propentropypsys}
Let $(\taue,\ue)$ be a strong entropy solution of \eqref{p_sys_eps}
and $(\taub,\ub)$ be a smooth solution of \eqref{p_sys0}. Then the
relative entropy $\eta^{\varepsilon}$, defined by
(\ref{entropyrelpsys}), satisfies the following evolution law:
\begin{align}
\pa_{t}\eta^{\varepsilon}(\taue,\ue|\taub,\ub)+ & \pa_{x}\psi(\taue,\ue|\taub,\ub)=\nonumber\\ 
& -\sigma(\ue-\ub)^{2}+\frac{1}{\sigma}p(\taue|\taub)\pa_{xx}p(\taub)
+\frac{\varepsilon^{2}}{\sigma}(\ue-\ub)\pa_{xt}p(\taub),\label{ineqentropyrelPsys}
\end{align}
where 
\begin{align}
& \psi(\tau,u|\taub,\ub) = (u-\ub)(p(\tau)-p(\taub)),\label{flux_etrp_relat}\\
& p(\tau|\taub)=p(\tau)-p(\taub)-p'(\taub)(\tau-\taub).\label{def-p-Tau-Taub}
\end{align}
\end{lemma}

Let us emphasize that equality \eqref{ineqentropyrelPsys} becomes an inequality as soon as the smoothness of solution $(\taue,\ue)$ is lost. The numerical counterpart is fully proved in the next section.

\begin{proof}
First, let us rewrite the parabolic system \eqref{p_sys0} such that we
get the same left hand side than for the scaled $p$-system
\eqref{p_sys_eps}. Then, (\ref{p_sys0}) reads equivalently as follows:
\begin{equation}\label{eqPsystbis}
\left\{\begin{aligned}
&\ds{\pa_{t}\taub-\pa_{x}\ub=0},\\
&\ds{\varepsilon^{2}\pa_{t}\ub+\pa_{x}p(\taub)=-\sigma\,\ub+\varepsilon^{2}\pa_{t}\ub.}
\end{aligned}\right.
\end{equation}

As a consequence, the derivative with respect to time of the relative entropy
\eqref{entropyrelpsys} satisfies the following sequence of equalities:
\begin{align*}
\pa_{t}\eta^{\varepsilon}(\taue,\ue|\taub,\ub)
&=\varepsilon^{2}(\ue-\ub)\pa_{t}(\ue-\ub)-p(\taue)\pa_{t}\taue+p(\taub)\pa_{t}\taub\\
& \quad \quad\quad
+p'(\taub)\pa_{t}\taub(\taue-\taub)+p(\taub)\pa_{t}(\taue-\taub)\\
&= -(\ue-\ub)\pa_{x}\left(p(\taue)-p(\taub)\right)-\sigma(\ue-\ub)^{2}
   -\varepsilon^{2}(\ue-\ub)\pa_{t}\ub\\
& \quad\quad\quad
-p(\taue)\pa_{x}\ue+p'(\taub)(\taue-\taub)\pa_{x}\ub+p(\taub)\pa_{x}\ue,\\ 
&=-\sigma(\ue-\ub)^{2}+\frac{\varepsilon^{2}}{\sigma}(\ue-\ub)\pa_{xt}p(\taub)\\ 
& \quad\quad\quad
  -\pa_{x}\Big(\left(p(\taue)-p(\taub)\right)(\ue-\ub)\Big)-p(\taue|\taub)\pa_{x}\ub.
\end{align*}
The expected result directly comes from $-\sigma\ub=\pa_{x}\ub$ to
write $\pa_{x}\ub = -\frac{1}{\sigma}\pa_{xx}p(\taub)$. The proof is thus completed.
\end{proof}

From now on, let us establish a technical result satisfied by the
relative internal energy $P(\tau|\taub)$, defined by (\ref{def-P-Tau-Taub}).
\begin{lemma}\label{lemma:p}
Assume that the pressure function $p(\tau)$ satisfies the conditions
(\ref{hyp_p}). Then there exists two positive constants, $C$ and $C'$,
such that for all $\tau\geq c>0$ and $\taub\geq c>0$, we have
\begin{equation}
|p(\tau|\taub)|\leq C'(\tau-\taub)^{2}\leq -C \, P(\tau|\taub).
\end{equation}
where $P(\tau|\taub)$ and $p(\tau|\taub)$ are respectively defined by
(\ref{def-P-Tau-Taub}) and (\ref{def-p-Tau-Taub}).
\end{lemma}

\begin{proof}
Since $p$ belongs to $C^2(\R_+^*)$, by definition of $p(\tau|\taub)$ and
$P(\tau|\taub)$, we immediately get
\begin{equation*}
\begin{aligned}
& p(\tau|\taub)=(\tau-\taub)^2\int_0^1 
     (1-s)p''\left(\taub+s(\tau-\taub)\right)\,ds,\\
& P(\tau|\taub)=(\tau-\taub)^2\int_0^1 
     (1-s)p'\left(\taub+s(\tau-\taub)\right)\,ds.
\end{aligned}
\end{equation*}
Because of the smoothness of $p$, there exists a positive constant $C'$
such that $|p''(\taub+s(\tau-\taub))|\leq 2C'$ for all $s\in(0,1)$. As
a consequence, we obtain
$$
|p(\tau|\taub)|\leq C'(\tau-\taub)^2.
$$

Moreover, the condition (\ref{hyp_p}) imposes the existence of a
positive constant $m$ such that $p'(\taub+s(\tau-\taub))\leq -2m$ for
all $s\in(0,1)$. Then we have
$$
-P(\tau|\taub)\ge m(\tau-\taub)^2.
$$
By considering $C=C'/m$, the proof is achieved.
\end{proof}

Arguing with these properties satisfied by the relative entropy, we
are now able to compare $(\taue,\ue)$, solution of (\ref{p_sys_eps}),
with $(\taub,\ub)$, solution of (\ref{p_sys0}). To address such an
issue and according to the assumptions stated in \cite{Lattanzio2013}
(see also \cite{Nishihara1996}), we impose that the porous media
equation is given for admissible specific volumes $\taub\geq
c>0$. Moreover, the solutions of \eqref{p_sys0} are assumed to be
smooth, hence we can consider regularity on the pressure function
$(x,t)\mapsto p(\taub(x,t))$ and its derivatives.

In addition, we suppose that the systems \eqref{p_sys_eps} and
\eqref{p_sys0} are endowed with initial conditions such that the
following limits hold:
\begin{equation}
  \label{eq:CI}
  \begin{aligned}
   & \lim_{x\to \pm \infty} \taue(x,t) = \lim_{x\to \pm \infty}\taub(x,t) = \tau_{\pm},\\
   & \lim_{x\to \pm\infty}\ue(x,t)=\lim_{x\to \pm\infty}\ub(x,t) = 0,
  \end{aligned}
\end{equation}
where $\tau_{\pm}$ are positive constant specific volume.

Now, let us introduce the positive error estimate given by
\begin{equation}\label{defphi}
\phi^{\varepsilon}(t)=\int_{\R}\eta^{\varepsilon}(\taue,\ue|\taub,\ub)dx,
\end{equation}
to establish the expected convergence rate away from vanishing specific volume (see also \cite{Lattanzio2013}).

\begin{theorem}\label{thrmcvpsys}
Consider initial data $(\taub_0(x),\ub_0(x))$ for (\ref{p_sys0}) and
$(\taue_0(x),\ue_0(x))$ for (\ref{p_sys_eps}) such that
$\phi^{\varepsilon}(0)<+\infty$. Endowed with these initial data, let
$(\taub,\ub)$ be the smooth solution of \eqref{p_sys0} defined on
$Q_{T}=\R\times[0,T)$, and let $(\taue,\ue)$ be a strong entropy solution of
\eqref{p_sys_eps}. Let us assume that $\taub\geq c>0$. Moreover, let
us assume that there exists
a positive constant $K$ such that
$\|\pa_{xx}p(\taub)\|_{L^{\infty}(Q_{T})}\leq K$ and 
$\|\pa_{xt}p(\taub)\|_{L^{2}(Q_{T})}\leq K$. Then the following stability estimate holds:
\begin{equation}\label{stabpsyst}
\phi^{\varepsilon}(t)\leq Ce^{CT}(\phi^{\varepsilon}(0)+\varepsilon^{4}), \quad t\in[0,T),
\end{equation}
where $C$ is a constant depending on $\sigma$ and $p(\taub)$. Moreover, if
$\phi^{\varepsilon}(0)\rightarrow 0$ as $\varepsilon \rightarrow 0$,
then
\begin{equation}\label{convergencepsys}
\sup_{t\in[0,T)}\phi^{\varepsilon}(t)\rightarrow 0, \text{ as }
\varepsilon \rightarrow 0. 
\end{equation}
\end{theorem}

\begin{proof}
Arguing the limit assumptions (\ref{eq:CI}), we have
$\psi^{\varepsilon}(\taue,\ue|\taub,\ub)\rightarrow 0$ in the limit $x
\rightarrow \pm \infty$. As a consequence, the integration of
\eqref{ineqentropyrelPsys} over $\mathbb{R}\times[0,t]$, for all
$t<T$, gives
\begin{equation}\label{thrm-continu-1}
\begin{aligned}
\phi^{\varepsilon}(t)-\phi^{\varepsilon}(0)\leq
-\sigma\int_{0}^{t}\int_{\R}(\ue-\ub)^{2}dx\,ds+\frac{1}{\sigma}\int_{0}^{t}\int_{\R}
\pa_{xx}p(\taub)\,p(\taue|\taub)dx\,ds\\
+\frac{\varepsilon^{2}}{\sigma}\int_{0}^{t}\int_{\R}\pa_{xt}p(\taub)\,(\ue-\ub)dx\,ds.
\end{aligned}
\end{equation}

Now, we estimate the integrals within the above relation. First, by
Lemma \ref{lemma:p} and since $\|\pa_{xx}p(\taub)\|_{L^{\infty}}\leq
K$, there exists a positive constant, say $C$, such that we have
\begin{equation*}
\frac{1}{\sigma}\int_{0}^{t}\int_{\R}|\pa_{xx}p(\taub)\,p(\taue|\taub)|dx\,ds\leq
\frac{C}{\sigma}\int_{0}^{t}\phi^{\varepsilon}(s)\,ds.
\end{equation*}

Concerning the last integral in (\ref{thrm-continu-1}), applying
Cauchy-Schwarz and Young's inequalities together with the assumption
on $\|\pa_{xt}p(\taub)\|_{L^{2}(Q_{T})}\leq K$, we immediately obtain
\begin{equation*}
\begin{aligned}
\frac{\varepsilon^{2}}{\sigma}\int_{0}^{t}\int_{\R}|\pa_{xt}p(\taub)\,(\ue-\ub)|dx\,ds
& \leq \frac{\sigma}{2}\int_{0}^{t}\int_{\R}(\ue-\ub)^{2}dx\,ds
+\frac{\varepsilon^{4}}{2\,\sigma^{3}}\int_{0}^{t}\int_{\R}|\pa_{xt}p(\taub)|^{2}dx\,ds\\
& \leq \frac{\sigma}{2}\int_{0}^{t}\int_{\R}(\ue-\ub)^{2}dx\,ds+C\,\varepsilon^{4}. 
\end{aligned}
\end{equation*}

As a consequence, identity \eqref{thrm-continu-1} now reads
$$
\phi^{\varepsilon}(t)-\phi^{\varepsilon}(0)\leq
   -\frac{\sigma}{2}\int_{0}^{t}\int_{\R}(\ue-\ub)^{2}dx\,ds
   +\frac{C}{\sigma}\int_{0}^{t}\phi^{\varepsilon}(s)\,ds
   +C\,\varepsilon^{4},
$$
to get
$$ 
\phi^{\varepsilon}(t)\le\phi^{\varepsilon}(0)
   +\frac{C}{\sigma}\int_{0}^{t}\phi^{\varepsilon}(s)\,ds
   +C\,\varepsilon^{4}.
$$
The required estimation \eqref{stabpsyst} is then obtained by the
Gr\"onwall's inequality. The proof is thus completed.
\end{proof}

\section{Semi-discrete finite volume scheme and numerical convergence rate}
\label{sec:semi-discrete-finite}
In this section, our purpose concerns the evaluation of the
convergence rate where both solutions $\we$ and $\wb$ are approximated
by a semi-discrete scheme.

Let us consider a uniform mesh made of cells
$(x_{i-\frac{1}{2}},x_{i+\frac{1}{2}})_{i\in\Z}$ of constant size $\Delta
x$. Here, the discretization points are given by $x_i=i\Delta x$ for
all $i\in\Z$. On each cell $(x_{i-\frac{1}{2}},x_{i+\frac{1}{2}})$,
the solutions of (\ref{p_sys_eps}) are approximated by time dependent
piecewise constant function $w_i(t)={}^t(\tau_i(t),u_i(t))$. For the sake of
clarity in the notations, we omit the dependence on the parameter
$\varepsilon$. Next, these functions are evolved in time by adopting a
semi-discrete scheme. Here, the suggested semi-discrete scheme is base
on the standard HLL numerical flux (see \cite{Harten1983}).
Hence the semi-discrete in space numerical scheme, to approximate the
solutions of (\ref{p_sys_eps}), reads
\begin{equation}
  \label{eq:Hd}
\left\{  \begin{aligned}
   \dfrac{d}{dt} \tau_i &= \dfrac{1}{2\Delta x} \left(
      u_{i+1} - u_{i-1}\right) + \dfrac{\lambda}{2\Delta x}(\tau_{i+1}-2\tau_i+\tau_{i-1}
      ),\\
   \dfrac{d}{dt}u_i &=\dfrac{\lambda}{2\Delta x}(u_{i+1}-2u_i+u_{i-1})
    - \dfrac{1}{2\varepsilon^2 \Delta x}(p(\tau_{i+1}) -
    p(\tau_{i-1})) -\dfrac{\sigma}{\varepsilon^2}u_i,
  \end{aligned}\right.
\end{equation}
where we have set 
\begin{equation}\label{defi_lambda}
\lambda=\sup_{t\in(0,T)}\max_{i\in \mathbb Z}(\sqrt{-p'(\tau_i)}).
\end{equation}
Let us underline that, as soon as $\varepsilon$ goes to zero,
the adopted semi-discrete finite volume scheme turns out to be
consistent with the porous media equation \eqref{p_sys0} (AC according
to the definition stated in the introduction). As a
consequence, the pair $\wb_i={}^t(\taub_i(t),\ub_i(t))$, to approximate the
solutions of (\ref{p_sys0}), are evolved in time as follows:
 \begin{equation}
      \label{eq:Pd}
\left\{      \begin{aligned}
        \dfrac{d}{dt} \bar\tau_i &= \dfrac{1}{2\Delta x} (
          \bar u_{i+1} - \bar u_{i-1} )+ \dfrac{\lambda}{2\,\Delta x}(\bar\tau_{i+1}-2 \bar\tau_i+\bar\tau_{i-1}
          ),\\
        \sigma \bar u_i &=-\dfrac{p(\bar\tau_{i+1}) -
        p(\bar\tau_{i-1})}{2\Delta x}
      \end{aligned}\right.
    \end{equation}

We now analyze the convergence from $(\tau_i,u_i)$ to
$(\taub_i,\ub_i)$ as $\varepsilon$ tends to zero. First, let us impose
the limit condition (\ref{eq:CI}) to be imposed to the approximate
solution as follows:
\begin{equation}
  \label{eq:CI_dis}
  \begin{aligned}
    \lim_{i\to \pm \infty} \tau_i &= \lim_{i\to \pm \infty} \bar \tau_i
    = \tau_{\pm},\\
    \lim_{i\to \pm \infty} u_i &= \lim_{i\to \pm \infty} \bar u_i
    = 0.
  \end{aligned}
\end{equation}

Next, to simplify the forthcoming estimations, we introduce several
semi-discrete norms. Let $v(t)=(v_i(t))_{i\in \mathbb Z}$ a function
of time $t\in[0,T)$ piecewise constant on cells
  $(x_{i-\frac{1}{2}},x_{i+\frac{1}{2}})$. Then we define
  \begin{align}
       &\|D_{x} v\|_{L^\infty(Q_T)} = \sup_{ t\in[0,T)} \sup_{i\in
        \mathbb Z} \left|\dfrac{v_{i+1} - v_i }{\Delta
          x} \right|,\nonumber\\
     & \|\tDxx v\|_{L^\infty(Q_T)} = \sup_{ t\in[0,T)} \sup_{i\in
        \mathbb Z} \left|\dfrac{v_{i+2} - 2v_i + v_{i-2}}{(2\Delta
          x)^2} \right|,\label{normeLinfDxxv}\\
     & \|\Dxx v\|_{L^{\infty}(Q_{T})}=\sup_{t\in[0,T)}\sup_{i\in\mathbb{Z}}\left|\frac{v_{i+1}-2v_{i}+v_{i-1}}{(\Delta x)^{2}}\right|, \label{normeLinftildeDxxv}\\    
      &\|\tDtx v\|_{L^2(Q_T)} = \left( \int_0^t \sumi \Delta x \;
        \left | \dfrac{d}{dt}\left(\dfrac{ v_{i+1} - v_{i-1}}{2\Delta x} \right)\right|^2 (s)
        ds\right)^{1/2},\label{normeL2Dtxv}\\
      &\|\Dxx  v\|_{L^2(Q_T)} = \left( \int_0^t \sumi \Delta x \;
        \left | \dfrac{v_{i+1} - 2v_i +
              v_{i-1}}{(\Delta x)^2} \right|^2 (s)
        ds\right)^{1/2}\label{normeL2Dxxv},
  \end{align}
where $Q_T=\mathbb R \times [0,T)$.

We adopt the approach introduced by Lattanzio and Tzavaras
\cite{Lattanzio2013} to the semi-discrete scheme (\ref{eq:Hd}). As a
first step, according to the definition of the relative entropy given
by (\ref{entropyrelpsys}), we now set
\begin{equation}    
  \label{eq:ERD}
  \begin{aligned}
    \eta_i^\varepsilon(t)&=\eta^\varepsilon(\tau_i, u_i|
    \bar \tau_i, \bar u_i)(t)\\
    &=\dfrac{\varepsilon^2}{2}(u_i(t)-\bar u_i(t))^2 - P(\tau_i(t)|\bar \tau_i(t)).
  \end{aligned}
\end{equation}

Mimicking the continuous framework, we introduce $\phi^\varepsilon(t)$
to denote the discrete space
integral of $\eta_i^\varepsilon(t)$ as follows:
\begin{equation}
  \label{eq:phieddef}
  \phi^\varepsilon(t) = \sum_{i \in \mathbb Z} \Delta x\; \eta_i^\varepsilon(t).
\end{equation}
Without ambiguity and for the sake of clarity, the time dependence is omitted in the sequel.

Now, we give our main result.
\begin{theorem}\label{thm:2}
  Let $\wb_i(t)=(\bar \tau_i(t), \bar u_i(t))_{i \in \mathbb Z}$ be
    a smooth solution of (\ref{eq:Pd}) away from zero, defined on $Q_T= \mathbb R
    \times [0,T)$. We assume the existence of a positive constant
      $K<+\infty$ such that the following estimations are satisfied:
    \begin{align}
      \label{eq:normD1}
&      \|\tDtx p(\bar \tau)\|_{L^2(Q_T)}\leq K,\qquad  \|\tDxx p(\bar
      \tau)\|_{L^\infty(Q_T)}\leq K\\
      \label{eq:normD2}
&      \|\Dxx\bar \tau\|_{L^\infty(Q_T)}\leq K, \qquad \|D_x\bar
      \tau\|_{L^\infty(Q_T)}\leq K,\qquad \|\Dxx \bar u\|_{L^2(Q_T)}\leq K.
    \end{align}
Let $w_i(t)=(\tau_i(t),u_i(t))_{i \in \mathbb Z}$ be a solution of
\eqref{eq:Hd}, away from zero, such that $\phi^\varepsilon(0) <+\infty$.
Then we have
    \begin{equation}
      \label{eq:rateD}
      \phi^\varepsilon(t) \leq Be^{BT}(\phi^\varepsilon(0) +
      \varepsilon^4), \quad t\in [0,T),
    \end{equation}
where $B$ is a positive constant which depends on $K$ and $\sigma$. Moreover if $\phi^\varepsilon(0)
\to 0$ as $\varepsilon\to 0$ then $\sup_{t\in[0,T)}\phi^\varepsilon(t)
  \to 0$ when $\varepsilon\to 0$.
\end{theorem}
Let us emphasize that the regularity conditions (\ref{eq:normD1})
exactly coincide with the smoothness imposed in Theorem
\ref{thrmcvpsys}. Here, because of the numerical viscous terms,
additional assumptions, stated in (\ref{eq:normD2}), must be imposed on
the approximate solution of the porous media equation. However such
conditions are not restrictive since solutions of the parabolic system
(\ref{p_sys0}), in general, come with enough smoothness.

Now, we turn to establish the above statement. To access such an
issue, we need three technical results. The first one is devoted to
exhibit the evolution law satisfied by the relative entropy
$\eta_i^\varepsilon$. We will see that this evolution law turns out to
be a discrete form of (\ref{ineqentropyrelPsys}) supplemented by
numerical viscosity. The two next Lemmas concern estimations of
the numerical viscous terms associated to the relative entropy.

Concerning the evolution law satisfied by $\eta_i^\varepsilon$, we
have the following result:
\begin{lemma} \label{lem1}
Let $(\bar \tau_i, \bar u_i)_{i \in \mathbb Z}$ be a smooth solution
of \eqref{eq:Pd} and let $(\tau_i,u_i)_{i \in \mathbb Z}$ be a
solution of \eqref{eq:Hd}. The relative entropy $\eta_i^\varepsilon$,
defined by (\ref{eq:ERD}), verifies the following evolution law:
  \begin{equation}
    \label{eq:dtERD}
    \begin{aligned}
      \dfrac{d\eta_i^\varepsilon}{dt} +\dfrac{1}{\Delta
        x}(\psi_{i+1/2}&-\psi_{i-1/2})= - \sigma(u_i-\bar u_i)^2\\
            &+ \dfrac{1}{\sigma} \dfrac{p(\bar \tau_{i+2}) - 2 p(\bar
        \tau_i) + p(\bar \tau_{i-2})}{(2\Delta x)^2}
      p(\tau_i|\bar\tau_i)\\
      &+ \dfrac{\varepsilon^2}{\sigma}(u_i-\bar
      u_i)\dfrac{d}{dt}\left( \dfrac{ p(\bar\tau_{i+1}) - p(\bar
          \tau_{i-1})}{2\Delta x}\right) \\
      &+ R_i^u + R_i^\tau,
    \end{aligned}
  \end{equation}
where $\psi_{i+1/2}$ corresponds to an approximation of the relative
entropy flux $\psi$ at the interface $x_{i+1/2}$ given by
\begin{equation}
  \label{eq:psiD}
  \psi_{i+1/2} =\dfrac 1 2 (u_i-\bar u_i)(p(\tau_{i+1}) - \bar p(\taub_{i+1})) +
  \dfrac 1 2 (u_{i+1} - \bar u_{i+1}) (p(\tau_i) - p(\bar \tau_i)),
\end{equation}
and the quantities $R_i^u$ and $R_i^\tau$ denote numerical residuals given by
\begin{equation}
  \label{eq:Ri}
  \begin{aligned}
    R_i^u&=\dfrac{\lambda \varepsilon^2}{2\Delta x} (u_i-\bar
  u_i)(u_{i+1}-2u_i+u_{i-1}),\\
    R_i^\tau&=-\dfrac{\lambda}{2\Delta x} \Big( (p(\tau_i) - p(\bar
      \tau_i)) (\tau_{i+1} -2\tau_i + \tau_{i-1})
    - (\tau_i-\bar
      \tau_i)p'(\bar \tau_i) (\bar \tau_{i+1} - 2\bar \tau_i + \bar
      \tau_{i-1})\Big).
  \end{aligned}
\end{equation}
\end{lemma}

From now on, we state estimations satisfied by both residuals $R_i^u$
and $R_i^\tau$.

\begin{lemma}\label{lem2}
Let $K<+\infty$ be a positive constant.
  Assume $\|\Dxx \bar u\|_{L^2(Q_T)}^2\le K$, then for all $\theta\in \mathbb R^*_+$, we have
\begin{equation}
  \label{eq:Riu}
  \int_0^t \sum_{i\in \mathbb Z} \Delta x\; R_i^u ds \leq
  \dfrac{\lambda\theta}{2}\int_0^t \sum_{i\in \mathbb Z} \Delta x\;
  (u_i-\bar u_i)^2 ds +
  \dfrac{\varepsilon^4\lambda\,\Delta x}{2\theta}\|\Dxx \bar u\|_{L^2(Q_T)}^2.
\end{equation}
\end{lemma}

\begin{lemma}\label{lem3}
Let $K<+\infty$ be a positive constant.
Let us assume $\|\Dxx\bar \tau\|_{L^\infty(Q_T)}\le K$ and $\|D_x
  \bar \tau\|_{L^\infty(Q_T)}<K$. Then there exists a positive
  constant $C$ such that
\begin{equation}
\label{eq:Ritau}
  \int_0^t \sum_{i\in \mathbb Z} \Delta x \; R_i^\tau ds \leq
  \lambda\Big( C\,\Delta x\|\Dxx \bar \tau\|_{L^\infty(Q_T)} +
C\|D_{x} \bar \tau\|_{L^\infty(Q_T)}\Big)\int_0^t \phi^\varepsilon(s) ds.
\end{equation}
\end{lemma}

Equipped with these three technical lemmas, we now establish our main result.
\begin{proof}[Proof of Theorem \ref{thm:2}]
Arguing Lemma \ref{lem2}, we evaluate the function $\phi^\varepsilon$
by a discrete integration in space of the equation \eqref{eq:dtERD}
and next, an integration in time over $[0,t)$. Since the limit
assumptions \eqref{eq:CI_dis} hold,
the relative entropy flux tends to 0 when $i \to \pm \infty$.
As a consequence, a straightforward computation gives
\begin{equation}
  \label{eq:phied}
  \begin{aligned}
    \phi^\varepsilon(t) -\phi^\varepsilon(0)&= -\sigma \int_0^t \sumi
    \Delta x\; (u_i - \bar u_i)^2 (s) ds \\
    &+ \dfrac 1 \sigma \int_0^t \sumi \Delta x\; \left(\dfrac{p(\bar
        \tau_{i+2}) - 2p(\bar \tau_i) + p (\bar \tau_{i-2}) }{(2\Delta
        x)^2}
      p(\tau_i|\bar\tau_i)\right) (s) ds\\
    &+ \dfrac{\varepsilon^2}{\sigma}\int_0^t \sumi \Delta x\; \left((u_i -
    \bar u_i) \dfrac{d}{dt} \left( \dfrac{p(\bar \tau_{i+1}) - p(\bar
        \tau_{i-1}) }{2 \Delta x} \right)\right) (s) ds\\
  &+ \int_0^t \sumi \Delta x\; ( R_i^u + R_i^\tau) (s) ds.
  \end{aligned}
\end{equation}
Now, we evaluate each term involved within the right-hand side. Let us
note that the second and third terms of \eqref{eq:phied} are nothing
but the discrete counterparts of the second and third terms in
\eqref{thrm-continu-1}. 

Concerning the second term of \eqref{eq:phied}, from the definition
\eqref{normeLinfDxxv} of $\|\tDxx p(\taub)\|_{L^{\infty}(Q_{T})}$ and
Lemma \ref{lemma:p}, the following estimation holds:
\begin{equation*}
\begin{aligned}
 \dfrac{1}{\sigma} \int_0^t \sumi \Delta x\; \left|\dfrac{p(\bar
          \tau_{i+2}) - 2p(\bar \tau_i) + p (\bar \tau_{i-2})
        }{(2\Delta x)^2} p(\tau_i|\bar\tau_i)\right|& (s) ds
      \leq \\ & \hspace*{-20mm}
      -\frac{C}{\sigma}\|\tDxx
      p(\taub)\|_{L^{\infty}(Q_{T})}\int_{0}^{t}\sumi\Delta x\,
      P(\tau_{i}|\taub_{i})(s)ds.
\end{aligned}
\end{equation*}
Because of definition \eqref{eq:ERD}, we have $-P(\tau_i|\taub_i)\le
\eta_i^\varepsilon$. As a consequence, by definition of
$\phi^\varepsilon$ given by \eqref{eq:phieddef}, we immediately obtain
\begin{gather}
 \dfrac{1}{\sigma} \int_0^t \sumi \Delta x\; \left|\dfrac{p(\bar
          \tau_{i+2}) - 2p(\bar \tau_i) + p (\bar \tau_{i-2})
        }{(2\Delta x)^2} p(\tau_i|\bar\tau_i)\right| (s) ds
      \leq 
      \frac{C}{\sigma}\|\tDxx p(\taub)\|_{L^{\infty}(Q_{T})}\int_{0}^{t}\phi^{\varepsilon}(s)ds.\label{eq:estimate_dxxp}
\end{gather}

Concerning the third term in \eqref{eq:phied}, we use the
Cauchy-Schwarz and Young's inequalities to get
 \begin{equation*}
    \begin{aligned}
      \dfrac{\varepsilon^2}{\sigma}\int_0^t &\sumi \Delta x\; \left|
        (u_i - \bar u_i) \dfrac{d}{dt} \left( \dfrac{p(\bar
            \tau_{i+1}) - p(\bar \tau_{i-1}) }{2 \Delta x}(s)
        \right)\right|  ds\\
      &\leq \dfrac{\sigma}{2} \int_0^t \sumi \Delta x\;
      |u_i -\bar u_i|^2 (s) ds 
      + \dfrac{\varepsilon^4}{2\sigma^3}\int_0^t \sumi \Delta x\; \left| 
          \dfrac{d}{dt} \left( \dfrac{p(\bar
              \tau_{i+1}) - p(\bar \tau_{i-1}) }{2 \Delta x}(s)
        \right)\right|^2 ds.
    \end{aligned}
  \end{equation*}
Involving the definition \eqref{normeL2Dtxv} of $\|\tDtx
p(\taub)\|_{L^{2}(Q_{T})}$, the following estimation holds:
\begin{equation}
    \label{eq:estimate_dtxp}
    \begin{aligned}
      \dfrac{\varepsilon^2}{\sigma}\int_0^t \sumi \Delta x\;
      &\left| (u_i - \bar u_i) \dfrac{d}{dt} \left( \dfrac{p(\bar
            \tau_{i+1}) - p(\bar \tau_{i-1}) }{2 \Delta x}
        \right)\right| (s) ds\\
      &\leq \dfrac{\sigma}{2} \int_0^t \sum_{i
        \in \mathbb Z} \Delta x\; (u_i-\bar u_i)^2 ds +
      \dfrac{\varepsilon^4}{2\sigma^3}\|\tDtx p(\bar
      \tau)\|_{L^2(Q_T)}^2.
    \end{aligned}
  \end{equation}

Now, the control of the numerical error terms $R_i^u$ and $R_i^\tau$ is established in Lemma
\ref{lem2} and Lemma \ref{lem3}, in order to have the estimations of
the last term in \eqref{eq:phied}. Accounting on the estimations
\eqref{eq:Riu}, \eqref{eq:Ritau}, \eqref{eq:estimate_dxxp} and
\eqref{eq:estimate_dtxp}, from the relation \eqref{eq:phied} we write 
\begin{equation}
  \label{eq:phied2}
    \begin{aligned}
    \phi^\varepsilon (t) \leq \phi^\varepsilon(0)& 
        + \left(\frac{\lambda\theta}{2}-\dfrac{\sigma}{2}\right)
        \int_0^t \sum_{i \in \mathbb Z} \Delta x\; (u_i- \bar
        u_i)^2 (s) ds\\
        &
        +\left(\dfrac{1}{2\sigma^3}\|\tDtx p(\bar \tau)\|_{L^2(Q_T)}^2 +
        \dfrac{\lambda\,\Delta x}{2\theta}\|\Dxx \bar
        u\|_{L^2(Q_T)}^2\right)\varepsilon^4\\
        &+\left( C\lambda\,\Delta x\|\Dxx \bar \tau\|_{L^\infty(Q_T)} + 
          C \lambda\|D_{x} \bar \tau\|_{L^\infty(Q_T)}+
          \dfrac{C}{\sigma } \|\tDxx p(\bar \tau)\|_{L^\infty(Q_T)}\right)\int_0^t
        \phi^\varepsilon(s) ds.
  \end{aligned}
\end{equation}
Let us fix $\theta\le\dfrac \sigma \lambda$ such that 
$\dfrac{\lambda\theta}{2}-\dfrac \sigma 2 \le 0$. Then we get 
\begin{equation}
  \label{eq:phied3}
  \begin{aligned}
    \phi^\varepsilon (t) &\leq \phi^\varepsilon(0) +
    \left(\dfrac{1}{2\sigma^3}\|\tDtx p(\bar \tau)\|_{L^2(Q_T)}^2 +
        \dfrac{\lambda\,\Delta x}{2\theta}\|\Dxx \bar
        u\|_{L^2(Q_T)}^2\right)\varepsilon^4\\ 
    &+ \left( \dfrac{C}{\sigma } \|\tDxx p(\bar
      \tau)\|_{L^\infty(Q_T)} +\lambda C\,\Delta x\|\Dxx \bar
      \tau\|_{L^\infty(Q_T)} + \dfrac{\lambda C}{2} \|D_{x} \bar
      \tau\|_{L^\infty(Q_T)}\right) \int_0^t \phi^\varepsilon(s) ds.
  \end{aligned}
\end{equation}
The expected estimation \eqref{eq:rateD} is a direct consequence of
the Gr\"onwall Lemma. The proof is thus completed.
\end{proof}

To conclude this section, we now give the proofs of the three
intermediate results.

\begin{proof}[Proof of Lemma \ref{lem1}]
From \eqref{eq:ERD}, the derivative with respect to time of the relative
entropy $\eta_i^\varepsilon$ reads
  \begin{equation}
    \label{eq:ERD1}
    \dfrac{d}{dt}\eta_i^\varepsilon = \varepsilon^2 (u_i -\bar u_i)
    \dfrac{d}{dt}(u_i-\bar u_i) -(p(\tau_i) - p(\bar \tau_i))\dfrac{d}{dt} \tau_i +
    (\tau_i-\bar \tau_i)p'(\bar\tau_i) \dfrac{d}{dt} \bar \tau_i .
  \end{equation}
Now, let us rewrite the second equation of \eqref{eq:Pd} as follows:
  \begin{equation}
    \label{eq:Pd2}
    \varepsilon^2 \dfrac{d}{dt} \bar u_i = \varepsilon^2 \dfrac{d}{dt}
    \bar u_i -\sigma \bar u_i -\dfrac{1}{2\Delta x}(p(\bar
    \tau_{i+1}) - p(\bar \tau_{i-1})).
  \end{equation}
From (\ref{eq:Hd}), since we have
$$
   \varepsilon^2\dfrac{d}{dt}u_i =\dfrac{\lambda\varepsilon^2}{2\Delta x}(u_{i+1}-2u_i+u_{i-1})
    - \dfrac{1}{2 \Delta x}(p(\tau_{i+1}) -
    p(\tau_{i-1})) -\sigma u_i,
$$
we obtain
$$
\begin{aligned}
\varepsilon^2\dfrac{d}{dt}(u_i-\ub_i) =&
  -\sigma (u_i-\bar u_i)
      -\varepsilon^2 \dfrac{d}{dt}\bar u_i
      - \dfrac{1}{2\Delta x}\Big(
      (p(\tau_{i+1}) - p( \tau_{i-1})) - (p(\bar\tau_{i+1}) - p(\bar
      \tau_{i-1}))\Big)\\
&      +\dfrac{\lambda\varepsilon^2}{2 \Delta x} (u_{i+1} - 2u_i + u_{i-1}).
\end{aligned}
$$
Plugging the above relation into \eqref{eq:ERD1} leads to
  \begin{equation}
    \label{eq:ERD2}
    \begin{aligned}
      \dfrac{d}{dt}\eta_i^\varepsilon = & -\sigma (u_i-\bar u_i)^2
      -\varepsilon^2 (u_i-\bar u_i) \dfrac{d}{dt}\bar u_i \\
      &- (u_i-\bar u_i)\dfrac{1}{2\Delta x}\Big(
      (p(\tau_{i+1}) - p( \tau_{i-1})) - (p(\bar\tau_{i+1}) - p(\bar
      \tau_{i-1}))\Big)\\
      &+\dfrac{\lambda\varepsilon^2}{2 \Delta x}(u_i-\bar u_i)
      (u_{i+1} - 2u_i + u_{i-1})\\
      &-(p(\tau_i) - p(\bar \tau_i)) \dfrac{d}{dt}\tau_i + (\tau_i -
      \bar\tau_i)p'(\bar \tau_i) \dfrac{d}{dt} \bar \tau_i.
    \end{aligned}
  \end{equation}
Next, we substitute $\dfrac{d}{dt} \tau_i$ and $\dfrac{d}{dt} \bar
\tau_i$ by their definitions, given by \eqref{eq:Hd} and \eqref{eq:Pd}, to obtain
  \begin{equation}
    \label{eq:ERD3}
    \begin{aligned}
      \dfrac{d}{dt}\eta_i^\varepsilon = & -\sigma (u_i-\bar u_i)^2
      -\varepsilon^2 (u_i-\bar u_i) \dfrac{d}{dt}\bar u_i \\
      &-\dfrac{1}{2\Delta x} \Big (p(\tau_{i+1}) -
        p(\taub_{i+1})) (u_i- \bar u_i)
      -(p(\tau_{i-1}) - p(\bar\tau_{i-1}))(u_i-\bar u_i)\\
     & \hspace*{12mm}+ (p(\tau_i) - p(\bar \tau_i)) 
      (u_{i+1} - u_{i-1})-(\tau_i- \bar \tau_i)p'(\bar
      \tau_i) (\bar u_{i+1} - \bar
      u_{i-1})\Big)\\
      &+ \dfrac{\lambda\varepsilon^2}{2 \Delta x}(u_i-\bar
      u_i)(u_{i+1} - 2u_i-u_{i-1})\\
      &-\dfrac{\lambda}{2\Delta x} \Big( (p(\tau_i) - p(\bar
      \tau_i)) (\tau_{i+1} -2\tau_i + \tau_{i-1})
    - (\tau_i-\bar
      \tau_i)p'(\bar \tau_i) (\bar \tau_{i+1} - 2\bar \tau_i + \bar
      \tau_{i-1})\Big),
    \end{aligned}
  \end{equation}
Let us remark that the two last terms are respectively the numerical error terms
  $R_i^u$ and $R_i^\tau$ defined in \eqref{eq:Ri}. Moreover, by
definition of $p(\tau_{i}|\bar \tau_{i})$, given by
(\ref{def-p-Tau-Taub}), the above relation rewrites as follows:
  \begin{equation}
    \label{eq:Q1}
    \begin{aligned}
      \dfrac{d}{dt}\eta_i^\varepsilon = & -\sigma (u_i-\bar u_i)^2
      -\varepsilon^2 (u_i-\bar u_i) \dfrac{d}{dt}\bar u_i \\
      &-\dfrac{1}{2\Delta x} \Big(
        (\bar u_{i+1} - \bar u_{i-1}) p(\tau_i |\bar \tau_i) 
        + (u_i -\bar u_i) (p(\tau_{i+1}) - p(\bar \tau_{i+1})) \\
        &\hspace*{12mm}+ (p(\tau_i) - p(\bar \tau_i)) (u_{i+1} - \bar u_{i+1})
        -  (u_{i-1} -\bar u_{i-1}) (p(\tau_i) - p(\bar \tau_i))\\
        &\hspace*{12mm}- (u_i -\bar u_i) (p(\tau_{i-1}) - p(\bar \tau_{i-1}))
      \Big)\\
      &+ R_i^u + R_i^\tau,
    \end{aligned}
  \end{equation}
Adopting the definition \eqref{eq:psiD}
  of the discrete relative entropy flux $\psi_{i+1/2}$, we directly obtain
  \begin{equation}
    \label{eq:ERD4}
    \begin{aligned}
      \dfrac{d}{dt}\eta_i^\varepsilon = & -\sigma (u_i-\bar u_i)^2
      -\varepsilon^2 (u_i-\bar u_i) \dfrac{d}{dt}\bar u_i \\
      &-\dfrac{1}{\Delta x}(\psi_{i+1/2} - \psi_{i-1/2})\\
      &- \dfrac{1}{2 \Delta x}(\bar u_{i+1}-\bar u_{i-1})p
      (\tau_i|\bar \tau_i)\\
      &+R_i^u + R_i^\tau.
    \end{aligned}
  \end{equation}
Finally, from the scheme definition \eqref{eq:Pd}, we deduce the
following two relations:
  \begin{equation*}
    \begin{aligned}
 &     \dfrac{d}{dt} \bar u_i = -\dfrac{1}{2\sigma\Delta x}
      \dfrac{d}{dt} (p(\bar \tau_{i+1}) - p(\bar \tau_{i-1})),\\
 &     \bar u_{i+1} - \bar u_{i-1} = -\dfrac{1}{2\sigma\Delta x}
      \left( p(\bar \tau_{i+2}) - 2p(\bar \tau_i) + p (\bar
        \tau_{i-2}) \right),
    \end{aligned}
  \end{equation*}
to recover the expected evolution law (\ref{eq:dtERD}). The proof is
thus achieved.
\end{proof}

\begin{proof}[Proof of Lemma \ref{lem2}]
Because of the definition \eqref{eq:Ri} of the residual $R_i^u$, we have
  \begin{equation}
    \label{eq:Riu1}
    \int_0^t \sumi \Delta x\; R_i^u (s) ds =
    \dfrac{\varepsilon^2\lambda}{2}
    \int_0^t \sumi (u_{i+1} -2u_i + u_{i-1}) (u_i -\bar u_i) (s) ds,
  \end{equation}
which equivalently rewrites as follows:
  \begin{equation}
    \label{eq:Riu2}
    \begin{aligned}
      \int_0^t \sumi \Delta x\; R_i^u (s) ds =&
      \dfrac{\varepsilon^2\lambda}{2} \int_0^t \sumi \Big( \bar
        u_{i+1} -2\bar u_i + \bar u_{i-1}\Big)(u_i -\bar u_i) (s)ds\\
&+
      \dfrac{\varepsilon^2\lambda}{2} \int_0^t \sumi
      \Big((u_{i+1}-\bar u_{i+1}) -2(u_i-\bar u_i) +
        (u_{i-1}-\bar u_{i-1}))\Big) (u_i -\bar u_i) (s) ds.
    \end{aligned}
  \end{equation}
Since $u_i$ and $\ub_i$ satisfy the assumption limit
(\ref{eq:CI_dis}), we immediately have
$$
\begin{aligned}
\sumi
      \Big((u_{i+1}-\bar u_{i+1}) -2(u_i-\bar u_i) +
        (u_{i-1}-\bar u_{i-1}))\Big)& (u_i -\bar u_i)
=\\
& -\sumi
\Big( (u_{i+1}-\ub_{i+1}) - (u_i-\ub_i) \Big)^2.
\end{aligned}
$$
As a consequence, we obtain the following inequality:
$$
      \int_0^t \sumi \Delta x\; R_i^u (s) ds \le
      \dfrac{\varepsilon^2\lambda}{2} \int_0^t \sumi \Big( \bar
        u_{i+1} -2\bar u_i + \bar u_{i-1}\Big)(u_i -\bar u_i) (s)ds,
$$
which rewrites
  \begin{equation}
    \label{eq:Riu4}
    \int_0^t \sumi \Delta x\; R_i^u (s) ds \leq
    \dfrac{\varepsilon^2\lambda\,\Delta x}{2}
    \int_0^t \sumi \sqrt{\Delta x} \;
    \dfrac{\bar u_{i+1} -2\bar u_i + \bar u_{i-1}}{(\Delta x)^{2}}\sqrt{\Delta x}
    (u_i - \bar u_i)ds.
  \end{equation}
 Combining again Cauchy-Schwarz and Young's inequalities gives, for
 all $\theta>0$,
  \begin{equation}
    \label{eq:Riu5}
    \begin{aligned}
      \int_0^t \sumi \Delta x\; R_i^u (s) ds &\leq \dfrac{\varepsilon^4
        \lambda\,\Delta x}{2\theta} \int_0^t \sumi \Delta x \; \left(
        \dfrac{\bar u_{i+1} -2\bar u_i + \bar u_{i-1}}{(\Delta x)^{2}}
      \right)^2 ds \\
      &+ \dfrac{\lambda \theta}{2} \int_0^t \sumi \Delta
      x\; (u_i -\bar u_i)^2 ds.
    \end{aligned}
  \end{equation}
  Finally, the definition \eqref{normeL2Dxxv} of $\|\Dxx
  \ub\|_{L^{2}(Q_{T})}$ leads to the required inequality (\ref{eq:Riu}).
\end{proof}

\begin{proof}[Proof of Lemma \ref{lem3}]
First, arguing the definition of $p(\tau_i|\bar \tau_i)$, given by
(\ref{def-p-Tau-Taub}), a straightforward computation leads to the
following reformulation of $R_i^\tau$:
  \begin{equation}
    \label{eq:Ritau1}
    \begin{aligned}
      R_i^\tau =& -\dfrac{\lambda}{2\Delta x} \Big( p(\tau_i |\bar
        \tau_i) (\bar \tau_{i+1} -2\bar\tau_i + \bar \tau_{i-1}) \Big)\\
        &+\dfrac{\lambda}{2\Delta x} 
              \Big((p(\tau_i) - p(\bar \tau_i)) \left( (\tau_{i+1} - \bar
          \tau_{i+1}) - 2(\tau_i - \bar \tau_i) + (\tau_{i-1} - \bar
          \tau_{i-1}) \right) \Big),
    \end{aligned}
  \end{equation}
to get
\begin{equation}
    \label{eq:Ritau2}
      \int_0^t \sumi \Delta x\; R_i^\tau ds= T_1 +T_2,
  \end{equation}
  where we have set
  \begin{eqnarray}
    \label{eq:T}
    &&T_1 = -\dfrac{\lambda}{2} \int_0^t \sumi p(\tau_i | \bar \tau_i)
    (\bar\tau_{i+1} - 2 \bar \tau_i + \bar \tau_{i-1}) ds,\\
    &&T_2 = -\dfrac{\lambda}{2} \int_0^t \sumi (p(\tau_i) - p(\bar
    \tau_i))
    \Big( (\tau_{i+1} - \bar \tau_{i+1}) -2(\tau_i - \bar \tau_i)
      +(\tau_{i-1}-\bar \tau_{i-1}) \Big) ds.
  \end{eqnarray}
We first estimate $T_1$. Thanks to Lemma \ref{lemma:p}, we write
  \begin{equation}
    \label{eq:T11}
    T_1 \leq -\frac{\Delta x\,\lambda C}{2}\int_0^t \sumi \Delta x\;
    P(\tau_i | \bar \tau_i)
   \left| \dfrac{\bar\tau_{i+1} - 2 \bar \tau_i + \bar \tau_{i-1}}{ (\Delta x)^{2}}\right| ds.
  \end{equation}
Since we have $-P(\tau_i|\taub_i)\le \eta_i^\varepsilon$ and
$\|\Dxx\taub\|_{L^{\infty}(Q_{T})}$ is bounded, we easily obtain
  \begin{equation}
    \label{eq:T12}
    T_1 \leq \Delta x\,\lambda C \|\Dxx\bar \tau\|_{L^\infty(Q_T)}\int_0^t
    \phi^\varepsilon(s) ds.
  \end{equation}

  Now, let focus on $T_2$.
  By a discrete integration by parts, we directly get
  \begin{equation*}
      T_2 = \dfrac{\lambda}{2}\int_0^t \sumi
      \Big((p(\tau_{i+1})-p(\bar \tau_{i+1})) - (p(\tau_{i})-p(\bar
        \tau_{i})) \Big) \Big((\tau_{i+1}-\bar \tau_{i+1}) -
        (\tau_{i}-\bar
        \tau_{i}) \Big),
\end{equation*}
to write
\begin{equation*}
\begin{aligned}
     T_2 = &\dfrac{\lambda}{2}\int_0^t \sumi
      (p(\tau_{i+1})-p( \tau_{i})) 
      \Big((\tau_{i+1}-\bar \tau_{i+1}) -
        (\tau_{i}-\bar
        \tau_{i}) \Big)\\
      &-\dfrac{\lambda}{2}\int_0^t \sumi
      (p(\bar\tau_{i+1})-p(\bar
        \tau_{i})) \Big((\tau_{i+1}-\bar \tau_{i+1}) -
        (\tau_{i}-\bar
        \tau_{i}) \Big).
    \end{aligned}
  \end{equation*}
With some abuse in the notations, we introduce
\begin{equation*}
\frac{p(\tau_{i+1})-p(\tau_{i})}{\tau_{i+1}-\tau_{i}}(\tau_{i+1}-\tau_{i})=\left\{\begin{array}{ll}
p(\tau_{i+1})-p(\tau_{i}) &\text{ if }  \tau_{i+1}-\tau_{i}\neq 0,\\
0& \text{ otherwise,}
\end{array}\right.
\end{equation*}
to rewrite $T_2$ as follows: 
  \begin{equation}
    \label{eq:T22}
    \begin{aligned}
      T_2
      =& \dfrac{\lambda}{2}\int_0^t \sumi
      \dfrac{p(\tau_{i+1})-p( \tau_{i})}{\tau_{i+1}-\tau_i}
      \Big((\tau_{i+1}-\bar \tau_{i+1}) -
        (\tau_{i}-\bar
        \tau_{i}) \Big)(\tau_{i+1}-\tau_i)ds\\
      &-\dfrac{\lambda}{2}\int_0^t \sumi
      \dfrac{p(\bar\tau_{i+1})-p(\bar
        \tau_{i})}{\bar \tau_{i+1} -\bar \tau_i}
      \Big((\tau_{i+1}-\bar \tau_{i+1}) -
        (\tau_{i}-\bar
        \tau_{i}) \Big)(\bar \tau_{i+1} -\bar \tau_i)ds.
    \end{aligned}
  \end{equation}
We notice that
  \begin{equation*}
    \begin{aligned}
      \Big((\tau_{i+1}-\bar \tau_{i+1}) - (\tau_{i}-\bar \tau_{i})
      \Big)(\tau_{i+1}-\tau_i) 
      =&\Big( (\tau_{i+1} -\bar \tau_{i+1})
        - (\tau_i - \bar \tau_i)\Big)^2 \\
      &+(\bar \tau_{i+1} - \bar
      \tau_i) \Big( (\tau_{i+1} -\bar \tau_{i+1}) - (\tau_i - \bar
        \tau_i)\Big),
    \end{aligned}
  \end{equation*}
so that $T_2$ now reads
    \begin{equation}
    \label{eq:T23}
    \begin{aligned}
      T_2
      = &\dfrac{\lambda}{2}\int_0^t \sumi
      \dfrac{p(\tau_{i+1})-p( \tau_{i})}{\tau_{i+1}-\tau_i}
      \Big( (\tau_{i+1} -\bar \tau_{i+1})
        - (\tau_i - \bar \tau_i)\Big)^2ds\\
      &+\dfrac{\lambda}{2}\int_0^t \sumi
      \dfrac{p(\tau_{i+1})-p( \tau_{i})}{\tau_{i+1}-\tau_i}
      (\bar \tau_{i+1} - \bar
      \tau_i) \Big( (\tau_{i+1} -\bar \tau_{i+1}) - (\tau_i - \bar
        \tau_i)\Big)\\
       &-\dfrac{\lambda}{2}\int_0^t \sumi
      \dfrac{p(\bar\tau_{i+1})-p(\bar
        \tau_{i})}{\bar \tau_{i+1} -\bar \tau_i}
      \Big((\tau_{i+1}-\bar \tau_{i+1}) -
        (\tau_{i}-\bar
        \tau_{i}) \Big)(\bar \tau_{i+1} -\bar \tau_i)ds.
    \end{aligned}
  \end{equation}
According to the assumption \eqref{hyp_p}, the pressure $p$ is a
decreasing function of $\tau$. As a consequence, the first term of
\eqref{eq:T23} is nonpositive. Hence we obtain
  \begin{equation}
    \label{eq:T24}
    T_2 \leq \dfrac{\lambda}{2}\int_0^t \sumi
    \left( 
      \dfrac{p(\tau_{i+1})-p( \tau_{i})}{\tau_{i+1}-\tau_i}
      -
      \dfrac{p(\bar\tau_{i+1})-p(\bar
        \tau_{i})}{\bar \tau_{i+1} -\bar \tau_i}
    \right)
    \Big((\tau_{i+1}-\bar \tau_{i+1}) -
        (\tau_{i}-\bar
        \tau_{i}) \Big)(\bar \tau_{i+1} -\bar \tau_i)ds.
  \end{equation}
  Under the assumption \eqref{eq:normD2} on
  $\|D_{x}\taub\|_{L^{\infty}(Q_{T})}$, the above relation becomes
  \begin{equation}
    \label{eq:T25}
    \begin{aligned}
      T_2 \leq \dfrac{\lambda}{2} \|D_x \bar
      \tau\|_{L^\infty(Q_T)}\int_0^t \sumi \Delta x\; &\left|
        \dfrac{p(\tau_{i+1})-p( \tau_{i})}{\tau_{i+1}-\tau_i} -
        \dfrac{p(\bar\tau_{i+1})-p(\bar \tau_{i})}{\bar \tau_{i+1}
          -\bar \tau_i} \right| \\
      &\times\left|(\tau_{i+1}-\bar \tau_{i+1}) -
        (\tau_{i}-\bar \tau_{i}) \right|ds.
    \end{aligned}
  \end{equation}
Now, let us emphasize that we have
  \begin{equation*}
    \left| 
    \dfrac{p(\tau_{i+1})-p( \tau_{i})}{\tau_{i+1}-\tau_i}
      -
      \dfrac{p(\bar\tau_{i+1})-p(\bar
        \tau_{i})}{\bar \tau_{i+1} -\bar \tau_i}
    \right| 
    \leq \int_0^1 \Big|p'(\tau_i + z(\tau_{i+1}-\tau_i)) - p'(\bar \tau_i
    + z (\bar \tau_{i+1} - \bar \tau_i)))\Big|dz.
  \end{equation*}
  Since $p\in\mathcal{C}^{2}(\mathbb{R}_{+}^{*})$, the function
  $p'$ is Lipschitz continuous with a Lipschitz constant $D$. Then the
  following sequence of inequalities holds:
  \begin{equation*}
    \begin{aligned}
      \left| \dfrac{p(\tau_{i+1})-p( \tau_{i})}{\tau_{i+1}-\tau_i} -
        \dfrac{p(\bar\tau_{i+1})-p(\bar \tau_{i})}{\bar \tau_{i+1}
          -\bar \tau_i} \right| 
      &\leq D\int_0^1 \Big| (\tau_i +
         z(\tau_{i+1} - \tau_i ))-(\bar
         \tau_i + z (\bar \tau_{i+1} - \bar \tau_i) \Big| dz,\\
       &\leq D\int_0^1 \Big((1-z)|\tau_i-\bar \tau_i| + z|\tau_{i+1}
       -\bar\tau_{i+1}|\Big)dz,\\
       &\leq \dfrac{D}{2} \left( |\tau_i-\bar \tau_i|+ 
         |\tau_{i+1}-\bar\tau_{i+1}| \right).
    \end{aligned}
  \end{equation*}
  Plugging this estimation into \eqref{eq:T25} gives
    \begin{equation*}
    \begin{aligned}
      T_2 &\leq \dfrac{\lambda}{2} \dfrac{D}{2} \|D_x \bar
      \tau\|_{L^\infty(Q_T)}\int_0^t \sumi \Delta x\;
      \left(|\tau_{i+1}-\bar \tau_{i+1})|+ |\tau_{i}-\bar
        \tau_{i})| \right)^{2}ds,\\
      &\leq {\lambda D} \|D_x \bar
      \tau\|_{L^\infty(Q_T)}\int_0^t \sumi \Delta x\;
      |\tau_{i}-\bar
        \tau_{i}|^2ds.
\end{aligned}
\end{equation*}
By Lemma \ref{lemma:p}, there exists a positive constant $C$ such that
$|\tau_i-\taub_i|^2\le -CP(\tau_i|\taub_i)\le C\eta_i^\varepsilon$. As
a consequence, there exists a constant, once again denoted $C$, such
that we have 
\begin{equation}
    \label{eq:T26}
       T_2 \leq \lambda C \|D_x \bar
      \tau\|_{L^\infty(Q_T)}\int_0^t \phi^\varepsilon (s) ds,
  \end{equation}
Both inequalities \eqref{eq:T12} and \eqref{eq:T26} complete the
estimation of $R_i^\tau$ and the proof is achieved.
\end{proof}

\section{Numerical illustrations}
In this section, we perform numerical experiments to attest the
relevance of the established convergence rate given by
(\ref{eq:rateD}). To address such an issue, we consider a fully
discrete scheme as proposed by Jin {\it et al.} in \cite{Jin1998}. This
scheme is based on a reformulation of system \eqref{p_sys_eps} as
follows:
\begin{equation*}\left\{\begin{aligned}
& \pa_{t}\tau-\pa_{x}u=0,\\
& \ds{\pa_{t}u+\pa_{x}p(\tau)=-\frac{1}{\varepsilon^{2}}\left(\sigma\,u+(1-\varepsilon^{2})\pa_{x}p(\tau)\right)}.
\end{aligned}\right.
\end{equation*}
Arguing this reformulation, a 2-step splitting technique is
adopted. During the first step, a purely convective and non-stiff
system is considered:
\begin{equation*}\left\{\begin{aligned}
&\pa_{t}\tau-\pa_{x}u=0,\\
&\pa_{t}u+\pa_{x}p(\tau)=0.\end{aligned}\right.
\end{equation*} 
Its solutions are approximated by adopting a classical HLL scheme
\cite{Harten1983}:
\begin{subequations}\label{scheme-JPT}
\begin{align}
\tau_{i}^{n+\frac{1}{2}}=\tau_{i}^{n}-\frac{\Delta t}{\Delta x}\left(\mathcal{F}_{i+\frac{1}{2}}^{\tau}-\mathcal{F}_{i-\frac{1}{2}}^{\tau}\right),\label{scheme-JPT-tau-1}\\
u_{i}^{n+\frac{1}{2}}=u_{i}^{n}-\frac{\Delta t}{\Delta
  x}\left(\mathcal{F}_{i+\frac{1}{2}}^{u}-\mathcal{F}_{i-\frac{1}{2}}^{u}\right),\label{scheme-JPT-u-1}
\end{align}
\end{subequations}
where the numerical fluxes are defined by
\begin{align*}
&\mathcal{F}_{i+\frac{1}{2}}^{\tau}=\frac{1}{2}(-u_{i}^{n}-u_{i+1}^{n})-\frac{\lambda}{2}(\tau_{i+1}^{n}-\tau_{i}^{n}),\\
&\mathcal{F}_{i+\frac{1}{2}}^{u}=\frac{1}{2}(p(\tau_{i}^{n})+p(\tau_{i+1}^{n}))-\frac{\lambda}{2}(u_{i+1}^{n}-u_{i}^{n}).
\end{align*}
It is well known that this scheme is stable under the CFL condition
$\ds{\frac{\Delta t}{\Delta x}\lambda\leq \frac{1}{2}}$, where $\lambda$ is defined by \eqref{defi_lambda}, which does
not depend on $\varepsilon$. Next, the stiff source term is treated
by a second step where the following system is discretized:
\begin{equation*}\left\{\begin{aligned}
&\pa_{t}\tau=0,\\
&\pa_{t}u=-\frac{1}{\varepsilon^{2}}\left(\sigma\,u+(1-\varepsilon^{2})\pa_{x}p(\tau)\right).\end{aligned}\right.
\end{equation*}
During this relaxation step, an implicit method is suggested in order
to obtain unconditional stability:
\begin{align*}
&\tau_{i}^{n+1}=\tau_{i}^{n+\frac{1}{2}},\\
& \frac{u_{i}^{n+1}-u_{i}^{n+\frac{1}{2}}}{\Delta t}=-\frac{1}{\varepsilon^{2}}\left(\sigma\,u_{i}^{n+1}+(1-\varepsilon^{2})\frac{p_{i+\frac{1}{2}}^{n+1}-p_{i-\frac{1}{2}}^{n+1}}{\Delta x}\right).
\end{align*}
As in \cite{Jin1998}, the nodal values are given by the following
centered discretization:
$$p_{i+\frac{1}{2}}^{n+1}=\frac{1}{2}\left(p(\tau_{i}^{n+1})+p(\tau_{i+1}^{n+1})\right).$$
Since $\tau_{i}^{n+1}=\tau_{i}^{n+\frac{1}{2}}$, let us emphasize that
$u_{i}^{n+1}$ can be computed explicitly from
$(\tau_{i}^{n},u_{i}^{n})_{i\in\mathbb{Z}}$. Finally, the relaxation
step can be written as
\begin{subequations}\label{scheme-JPT2}
\begin{align}
&\tau_{i}^{n+1}=\tau_{i}^{n+\frac{1}{2}},\label{scheme-JPT-tau-2}\\
& u_{i}^{n+1}=\left(\frac{\varepsilon^{2}}{\varepsilon^{2}+\sigma\,\Delta t}\right)u_{i}^{n+\frac{1}{2}}-\Delta t\left(\frac{1-\varepsilon2}{\Delta t\,\sigma+\varepsilon^{2}}\right)\frac{p(\tau_{i+1}^{n+\frac{1}{2}})-p(\tau_{i-1}^{n+\frac{1}{2}})}{2\,\Delta x}. \label{scheme-JPT-u-2}
\end{align}
\end{subequations}
We underline that this scheme corresponds to the semi-discrete framework introduced Section \ref{sec:semi-discrete-finite}. Indeed, combining \eqref{scheme-JPT-u-1} and \eqref{scheme-JPT-u-2}, we get
\begin{gather*}
u_{i}^{n+1}=u_{i}^{n}-\frac{\sigma\,\Delta t}{\varepsilon^{2}+\sigma\,\Delta t}u_{i}^{n}-\frac{\Delta t}{2\,\Delta x (\varepsilon^{2}+\Delta t\,\sigma)}\left(p(\tau_{i+1}^{n+1})-p(\tau_{i-1}^{n+1})\right)\\
+\frac{\Delta t\,\lambda}{2 \,\Delta x}\left(\frac{\varepsilon^{2}}{\varepsilon^{2}+\sigma\,\Delta t}\right)(u_{i+1}^{n}-2u_{i}^{n}+u_{i-1}^{n}).
\end{gather*}
Now, we fix $\ds{\overline{\Delta t}=\frac{\Delta
    t\,\varepsilon^{2}}{\varepsilon^{2}+\sigma\,\Delta t}}$, and we
note that this new time increment is consistent with 
$\Delta t$. We immediately remark that we recover \eqref{eq:Hd} as soon
as $\Delta t$ tends to zero.

Next, we consider the scheme (\ref{scheme-JPT})-(\ref{scheme-JPT2}) in
the limit of $\varepsilon$ to zero to approximate the solutions of the
parabolic problem \eqref{p_sys0}. We get the following scheme:
\begin{align*}
&\taub_{i}^{n+1}=\taub_{i}^{n}+\frac{\Delta t}{2\,\Delta x}\left(\ub_{i+1}^{n}-\ub_{i-1}^{n}\right)-\frac{\lambda\,\Delta t}{2\,\Delta x}(\taub_{i+1}^{n}-2\taub_{i}^{n}+\taub_{i-1}^{n}),\\
&\ub_{i}^{n+1}=-\frac{1}{2\,\sigma\,\Delta x}\left(p(\taub_{i+1}^{n+1})-p(\taub_{i-1}^{n+1})\right),
\end{align*}
which is an approximation of \eqref{p_sys0}.

We notice that this scheme is {AP} in the sense of the
definition given in the introduction. Indeed, its limit as
$\varepsilon\rightarrow 0$ is consistent (AC) with the parabolic problem
\eqref{p_sys0}, while its stability condition does not depend on
$\varepsilon$.

Equipped with this scheme, we now perform numerical experiments. We
approximate the solutions on the interval $(-4,4)$, and we consider
zero-flux boundary conditions. The final time of simulation is
$T=10^{-2}$. The friction coefficient is fixed to $\sigma=1$.

Concerning the pressure law, we adopt $p(\tau)=\tau^{-\gamma}$ where
the adiabatic coefficient is fixed to $1.4$.

We compute the approximate solutions of the hyperbolic system
\eqref{p_sys_eps} for different values of $\varepsilon$: $10^{-1}$,
$3.10^{-2}$, $10^{-2}$, $3.10^{-3}$, $10^{-3}$, $3.10^{-4}$,
$10^{-4}$, and different number of cells
$N=100,\,200,\,400,\,1600$. The two following initial data are considered:
\begin{itemize}
\item {Condition 1} (discontinuous):
\begin{equation}\label{CI1}
\tau_{0}(x)=\left\{\begin{array}{lcl} 2 & \text{ if } & x<0, \\ 1 & \text{ if } & x>0,
\end{array}\right.
\end{equation}
\item {Condition 2} (smooth):
\begin{equation}\label{CI2}
\tau_{0}(x)=\exp(-100x^{2})+1.
\end{equation}
\end{itemize}
Here, the initial velocity $u_{0}$ is computed to be compatible with
the discrete diffusive limit in order to avoid an initial layer:
$$ u_{i}^{0}=-\frac{1}{\sigma}\frac{p(\tau_{i+1}^{0})-p(\tau_{i-1}^{0})}{2\Delta x} . $$ 

We display, Figure \ref{fig-psys}, the discrete space integral of
the relative entropy $\phi^{\varepsilon}(T)$ with respect to
$\varepsilon$ in log scale for the $p$-system. We observe that both
for discontinuous and smooth initial condition, and for different
numbers of cells, the decay rate is always in $O(\varepsilon^{4})$,
which is in good agreement with Theorem~\ref{thm:2}. 

A natural extension of this work concerns the Goldstein-Taylor model, which reads
\begin{equation}
\left\{\begin{aligned}
&\pa_{t}\rhoe+\pa_{x}\je=0,\\
&\varepsilon^{2}\pa_{t}\je+\pa_{x}\rhoe=-\sigma\,\je,
\end{aligned}\right. \quad (x,t) \in \R \times \R_{+}.
\end{equation}
This system can be seen as a simplified two velocities kinetic model
in macroscopic variables (see for example \cite{Jin1999,Naldi2000}). In
the diffusion limit $\varepsilon\rightarrow 0$, the Goldstein-Taylor
model coincides with the heat equation given by
\begin{equation}\label{eqPGT}
\left\{\begin{aligned}
&\pa_{t}\rhob-\ds{\frac{1}{\sigma}}\pa_{xx}\rhob=0,\\
&\pa_{x}\rhob=-\sigma\,\jb, 
\end{aligned}\right. \quad (x,t) \in \R \times \R_{+}.
\end{equation}
Concerning this model, a direct adaptation of the numerical scheme
\eqref{scheme-JPT}-\eqref{scheme-JPT2} is suggested. The numerical
results are displayed Figure \ref{fig-tele}. As well as for the
$p$-system case, the convergence rate is also in $O(\varepsilon^4)$
which is in good agreement with convergence results given in
\cite{Lattanzio2013}.\\

\begin{figure}
\centering
\subfigure[Initial data 1]{\includegraphics[width=2.8in]{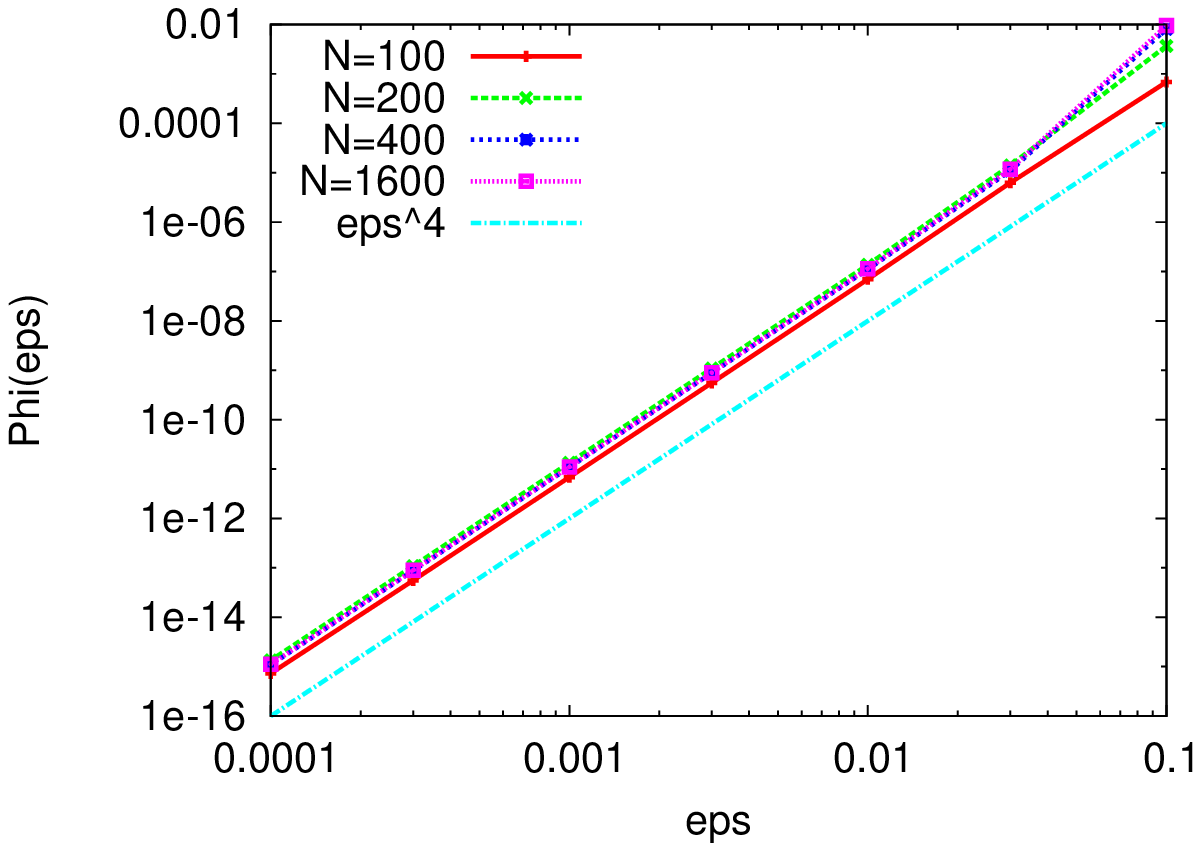}}
\subfigure[Initial data 2]{\includegraphics[width=2.8in]{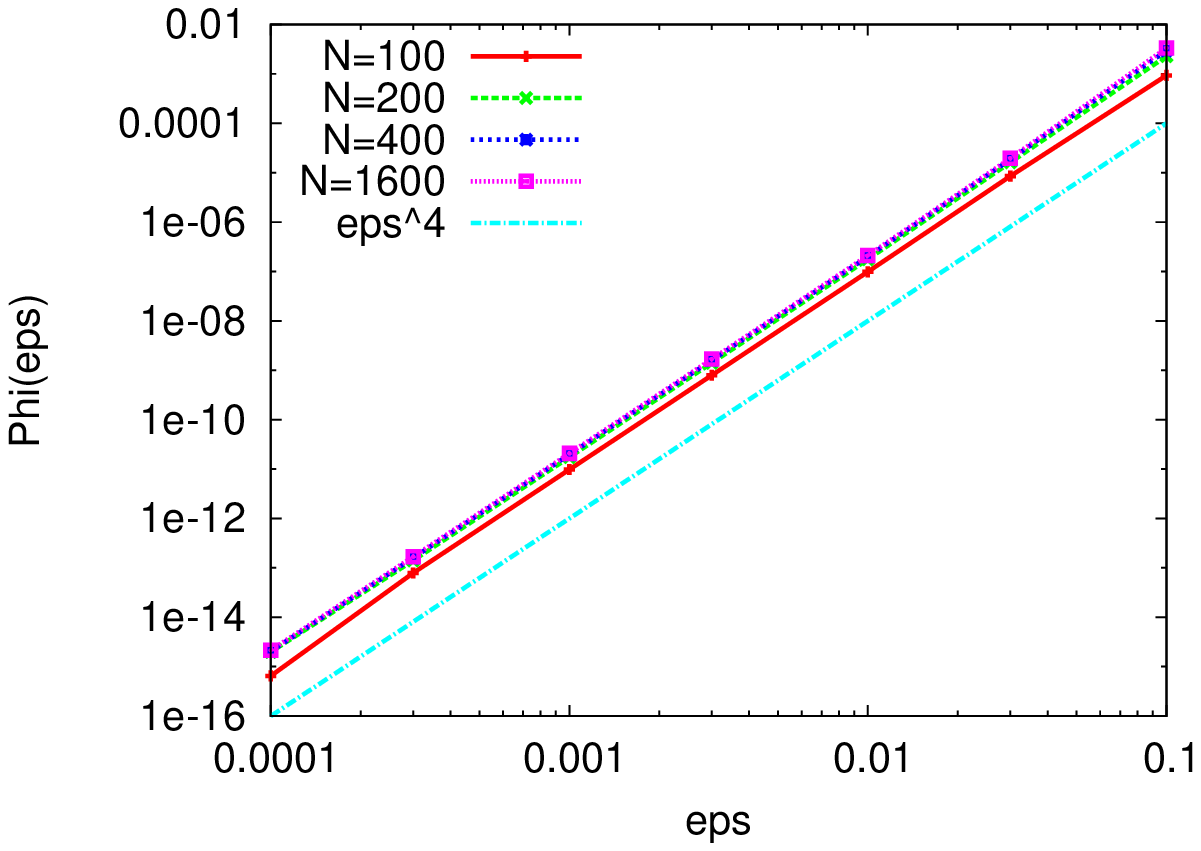}}
\caption{$p$-system: space integral of the relative entropy $\phi^{\varepsilon}$ with respect to $\varepsilon$ in log scale.}
\label{fig-psys}
\end{figure}

\begin{figure}
\centering
\subfigure[Initial data 1]{\includegraphics[width=2.8in]{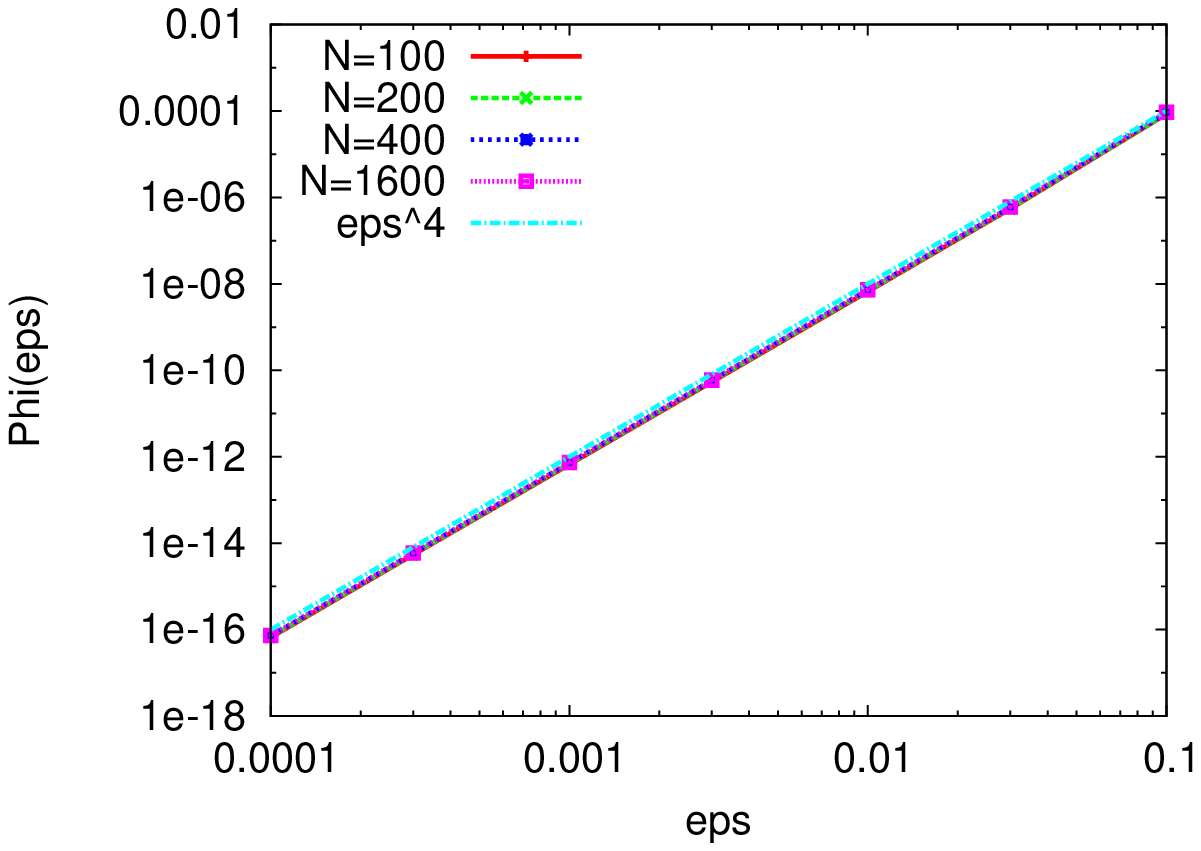}}
\subfigure[Initial data 2]{\includegraphics[width=2.8in]{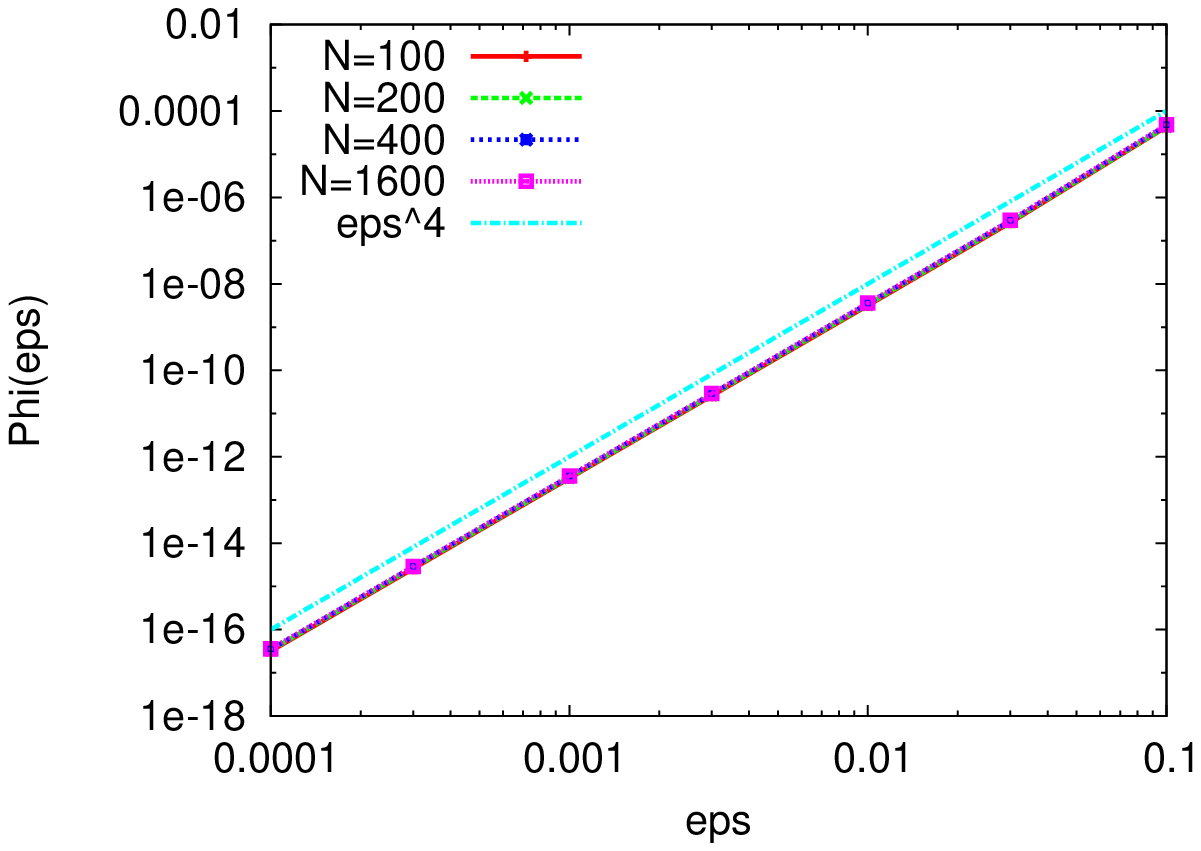}}
\caption{Goldstein-Taylor model: space integral of the relative entropy $\phi^{\varepsilon}$ with respect to $\varepsilon$ in log scale.}
\label{fig-tele}
\end{figure}

In \cite{Lattanzio2013}, the authors also apply their relative entropy method to other systems, leading to the same kind of estimates. To conclude this section, we also extend the numerical scheme \eqref{scheme-JPT}-\eqref{scheme-JPT2} to approximate the weak solutions of both isentropic Euler system and visco-elastic system with memory as considered in \cite{Lattanzio2013}. Concerning the isentropic Euler equation, the adopted scaled system is the following:
\begin{equation*}
\left\{\begin{aligned}
& \partial_{t}\rho^{\varepsilon}+\pa_{x}(\rho^{\varepsilon} u^{\varepsilon})=0,\\
& \partial_{t}(\rho^{\varepsilon} u^{\varepsilon})+\partial_{x}\left(\rho^{\varepsilon} (u^{\varepsilon})^2\right)+\frac{1}{\varepsilon^2}\partial_{x}p(\rho^{\varepsilon})=-\frac{\sigma}{\varepsilon^2}\rho^{\varepsilon} u^{\varepsilon},
\end{aligned}\right.  \quad (x,t) \in \R \times \R_{+}.
\end{equation*}
The corresponding asymptotic regime in the limit $\varepsilon\to 0$ is given by:
\begin{equation*}
\left\{\begin{aligned}
&\partial_{t}\overline{\rho}-\partial_{xx}p(\bar{\rho})=0,  \\
&\partial_{x}p(\bar{\rho})=-\sigma\bar{\rho}\bar{u},
\end{aligned}\right.\quad (x,t) \in \R \times \R_{+}.
\end{equation*}
Similarly, the visco-elastic system reads as follows:
\begin{equation*}
\left\{\begin{aligned}
& \partial_{t}u^{\varepsilon}-\pa_{x}v^{\varepsilon}=0,\\
& \partial_{t}v^{\varepsilon}-\partial_{x}\gamma(u^{\varepsilon}) -\partial_{x}z^{\varepsilon}=0,\\
& \partial_{t}z^{\varepsilon}-\frac{\mu}{\varepsilon^2}\partial_{x}v^{\varepsilon}=-\frac{\sigma}{\varepsilon^2}z^{\varepsilon},
\end{aligned}\right.  \quad (x,t) \in \R \times \R_{+},
\end{equation*}
where the asymptotic regime satisfies the following system:
\begin{equation*}
\left\{\begin{aligned}
&\partial_{t}\overline{u}-\partial_{x}\bar{v}=0,  \\
&\partial_{t}\bar{v}-\partial_{x}\gamma(\bar{u})=\mu \partial_{xx}\bar{v},\\
&\mu\partial_{x}\bar{v}=\sigma\bar{z},
\end{aligned}\right.\quad (x,t) \in \R \times \R_{+}.
\end{equation*}

The numerical results are displayed Figures \ref{fig-euler} and \ref{fig-visco}. We still observe a convergence rate in $O(\varepsilon^4)$, which is in good agreement with results established Theorem \ref{thm:2} (see also \cite{Lattanzio2013}).

\begin{figure}
\centering
\subfigure[Discontinuous initial data]{\includegraphics[width=2.8in]{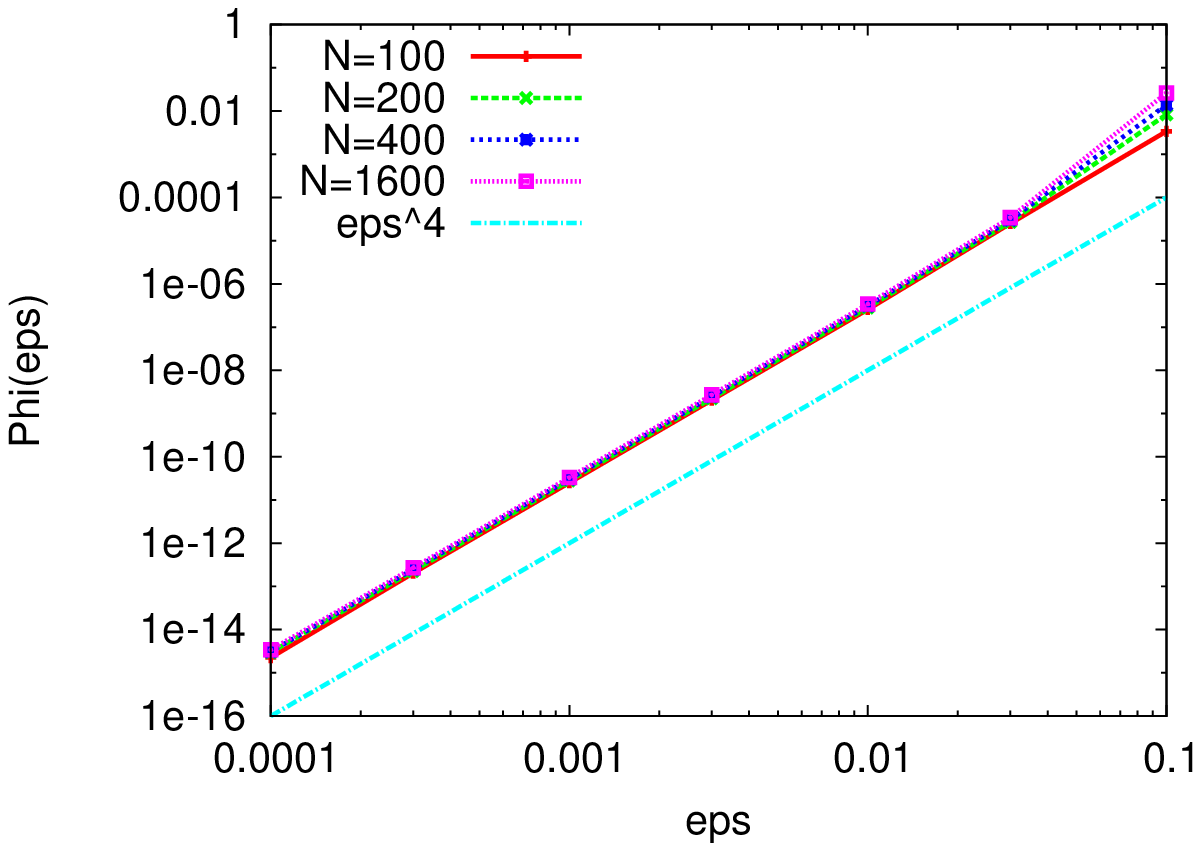}}
\subfigure[Smooth initial data]{\includegraphics[width=2.8in]{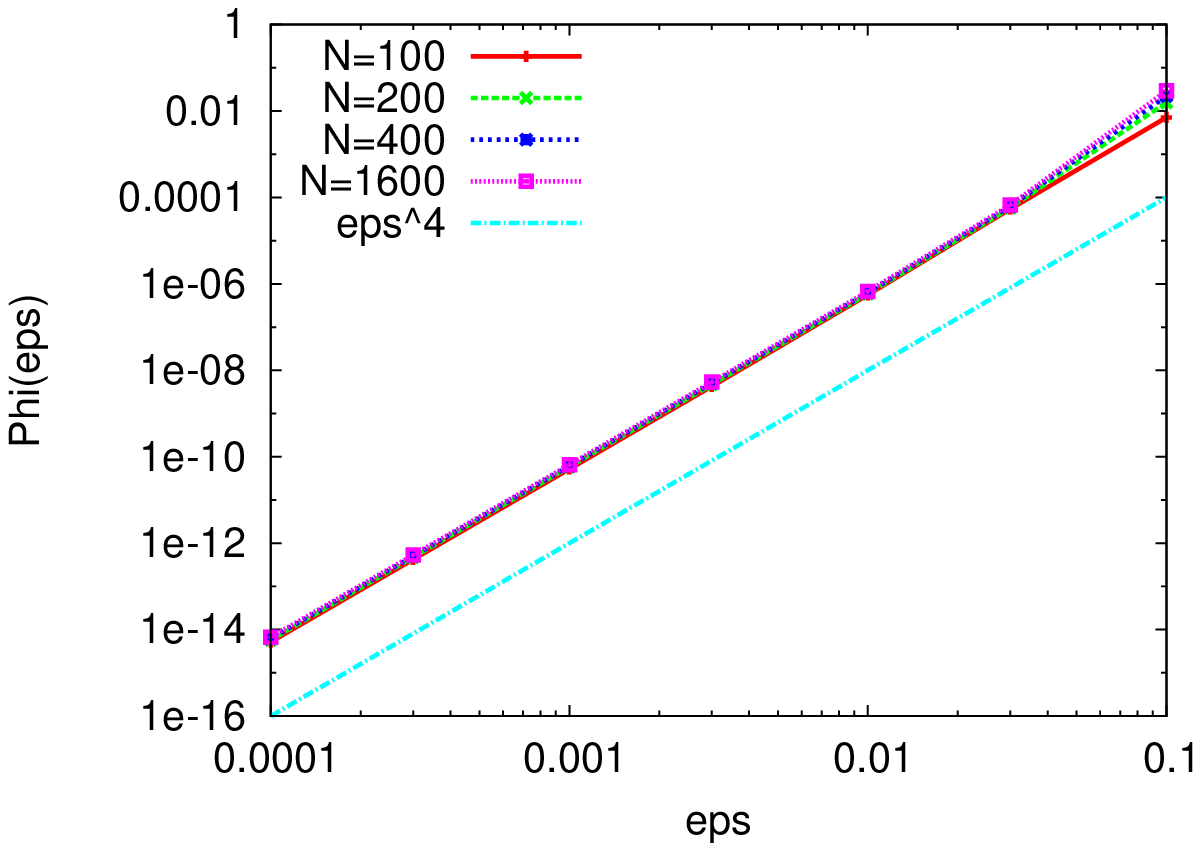}}
\caption{Isentropic Euler system: space integral of the relative entropy $\phi^{\varepsilon}$ with respect to $\varepsilon$ in log scale.}
\label{fig-euler}
\end{figure}

\begin{figure}
\centering
\subfigure[Discontinuous initial data]{\includegraphics[width=2.8in]{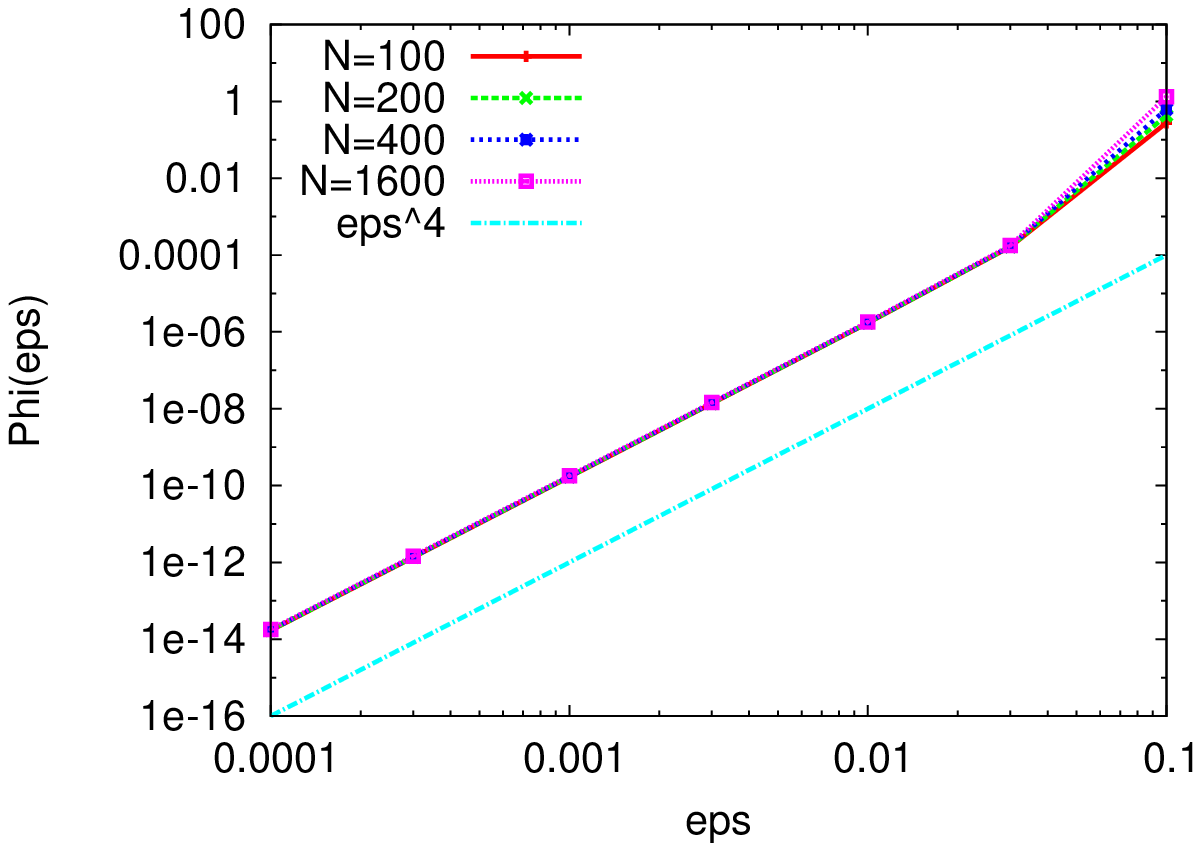}}
\subfigure[Smooth initial data]{\includegraphics[width=2.8in]{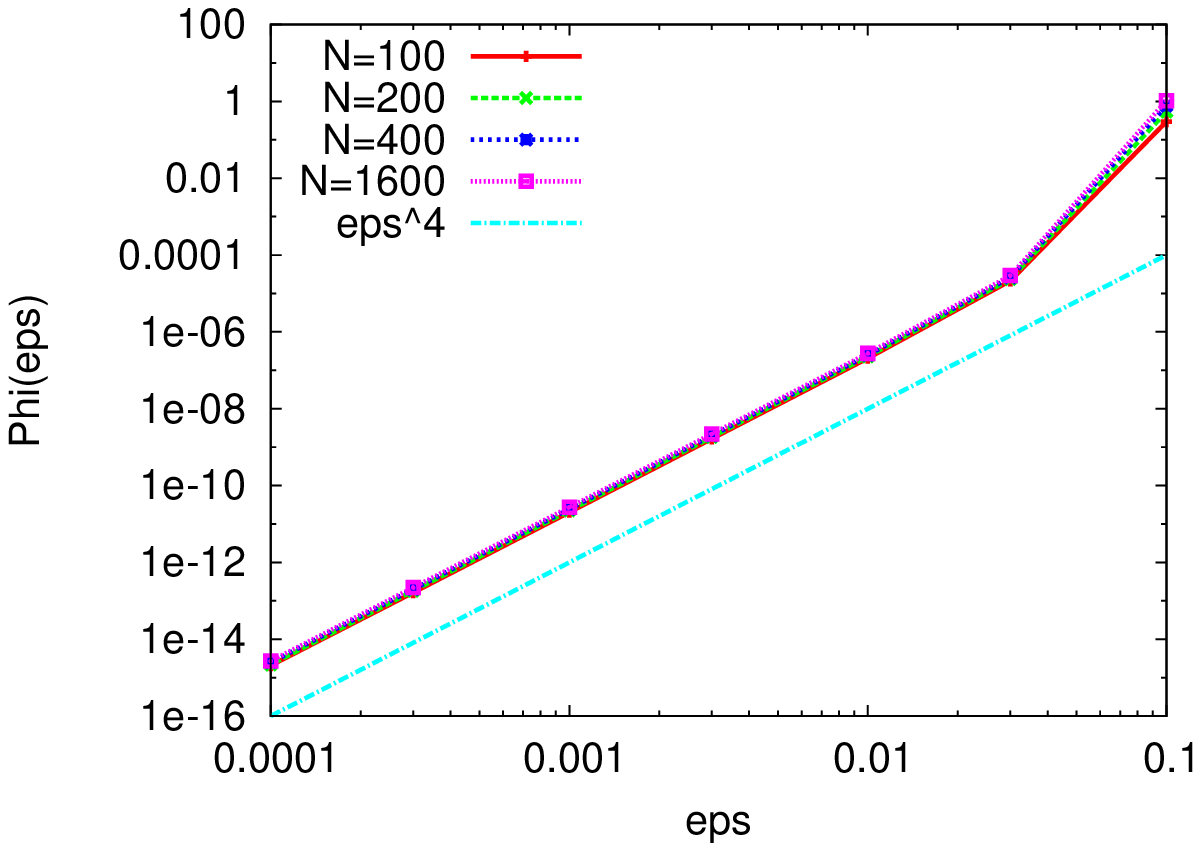}}
\caption{Visco-elastic model: space integral of the relative entropy $\phi^{\varepsilon}$ with respect to $\varepsilon$ in log scale.}
\label{fig-visco}
\end{figure}

\ \\
\noindent{\bf Acknowledgements.}
The authors thank the project ANR-12-IS01-0004 GeoNum and the project
ANR-14-CE25-0001 Achylles for their partial financial contributions.

\bibliographystyle{plain}
\bibliography{biblio_lim_diff}

\end{document}